\documentclass[a4paper,10pt]{amsart}
\usepackage{a4wide,amsmath,bbm,mathtools,mathrsfs,amsfonts,fancyhdr,color, ifthen, graphicx,stmaryrd, longtable,amssymb,latexsym,amsxtra,amscd,epic, array, multirow}

\DeclareMathOperator{\MC}{MC}

\DeclareMathOperator{\sign}{sign}
\DeclareMathOperator{\MH}{MH}
\DeclareMathOperator{\J}{\textbf{J}}
\DeclareMathOperator{\rk}{rk}

\DeclareMathOperator{\GL}{GL}

\DeclareMathOperator{\Sp}{Sp}

\DeclareMathOperator{\SO}{SO}

\DeclareMathOperator{\Sym}{Sym}

\DeclareMathOperator{\Sol}{Sol}
\newtheorem{defin}{Definition}[section]
\newtheorem{thm}[defin]{Theorem}
\newtheorem{lemma}[defin]{Lemma}

\newtheorem{prop}[defin]{Proposition}

\begin{document}
\thispagestyle{empty}
\flushleft
\title{Some fourth order CY-type operators with non symplectically rigid monodromy}
\author{Michael Bogner and Stefan Reiter}
\address{Michael Bogner, Institut f\"ur Mathematik, Johannes Gutenberg-Universit\"at Mainz, Staudingerweg 9, 55128 Mainz, Germany.}
\address{Stefan Reiter, Lehrstuhl IV f\"ur Mathematik, Universit\"at Bayreuth, 95440 Bayreuth}
\maketitle

\begin{abstract}
We study tuples of matrices with rigidity index two in $\Sp_4(\mathbb{C})$, which are potentially induced by differential operators of Calabi-Yau type. The constructions of those monodromy tuples via algebraic operations and middle convolutions and the related constructions on the level differential operators lead to previously known and new examples.
\end{abstract}

\section{Introduction}

Fourth order differential operators of Calabi-Yau type intend to describe periods of one-parameter families of Calabi-Yau threefolds over $\mathbb{P}^1$ minus a finite set of points which admit a point of maximally unipotent monodromy at the origin $z=0$. The most prominent example is the Picard-Fuchs operator for the mirror of a family of quintics in $\mathbb{P}^4$ which was investigated by Candelas et al. in \cite{Can}. Their amazing results concerning the prediction of Gromov-Witten invariants of the family was for many mathematicians a reason to study mirror symmetry. During the last years, a characterization of those differential operators from a purely differential-algebraic point of view as well as a collection of examples was established, see e.g. \cite{AESZ}, \cite{Alm} and \cite{Diss}. This leads to the notion of a \textit{CY-type differential operator}, see Definition \ref{CYdefin}. Those operators are irreducible, self-dual, Fuchsian and their local solutions underly further integrality conditions. The majority of known examples of order four was found by computer searches and is - from a geometric point of view - still purely understood. A step towards a geometric realization can be provided by checking whether the differential operator can be constructed from differential operators of lower order by operations which are known to preserve their so called \textit{geometric origin}, see e.g. \cite[Chapter II]{Andre} for more details. 
For a given differential operator $L$ with finite singular locus $S\subset \mathbb{P}^1$, this can be done by looking at the associated local system $\mathbb{L}$ whose sections are given by $\mathbb{L}(U)=\left\{f\in\mathcal{O}_{\mathbb{P}^1\setminus S}\mid L(f)=0\right\}.$ 
If this local system is physically rigid in the sense of \cite{katz96}, it can be constructed by a series of middle convolutions and tensor products of Kummer sheaves and hence inherits a geometric interpretation. Related constructions on the level of differential operators where established \cite{BoR}. 
Moreover, we showed in \cite{BoR} that symplectically rigid local systems which are potentially induced by differential operators of CY-type admit constructions which involve more general tensor operations. Carrying out these constructions on the level of differential operators explicitely, we were able to regain all known examples of CY-type operators which induce such a local system. In that spirit and by the notion of the \textit{index of rigidity} of a local system, this article can be seen as a continuation of \cite{BoR}, as it is devoted to potential local systems of CY-type operators with index two in $\Sp_4(\mathbb{C})$. 
As in \cite{BoR} or \cite{DRGalois}, we rather use a more explicit framework for our constructions. Given a local system $\mathbb{L}$ of rank $n$ on $\mathbb{P}^1\setminus S$, we get an associated tuple of $n\times n$-matrices $T=(T_1,\dots,T_{r+1})\in \GL_n(\mathbb{C})$ with $\prod_{i=1}^{r+1}T_i=\mathbbm{1}_n$ by the choice of an orientation on $\mathbb{P}^1$, a base point $x_0\in\mathbb{P}^1\setminus S$ and a set of generators $\gamma_1,\dots,\gamma_{r+1}\in\pi_1(\mathbb{P}^1\setminus S, x_0)$ such that $\gamma_1\circ\dots\circ\gamma_{r+1}$ is homotopic to the trivial path. All those underlying topological choices induce a equivalence of categories and hence enable us to work with the related tuple $T$, a so called \textit{M-tuple}. Throughout this article, we are often not specific concerning those underlying choices. As many properties of an M-tuple can be read off by the Jordan forms of its matrices, we also collect the tuple of Jordan forms associated to an M-tuple.

In the first section, we recall some operations on tuples of matrices and differential operators as established in \cite{DRGalois}, \cite{DRFuchs} and \cite{BoR}. Thereafter, we state possible tuples of Jordan forms which are induced CY-type operators with index two. In turns out, that each of those operators admits either three or four non-apparent singularities. Although we only find some exceptional examples in the case of three singularities, we provide a construction of a inducing family of differential operators for each of the other cases in the second section. The constructions rely on special differential operators of order two, so called \textit{Heun operators}. The third section is devoted to the detection of CY-type operators inside the families we constructed before. Due to our observations, we end up with a CY-type operator if the operator of order two we started with is of CY-type as well. All of the examples of suitable CY-type operators of order two we know are provided by algebraic pullbacks of hypergeometric differential operators and the majority of them even has a direct geometric interpretation as Picard-Fuchs operators of families of elliptic curves with four singular fibers, see \cite{Her}. In this sense, the resulting CY-type operator also inherits a geometric interpretation which should be subject of further studies. Via these methods, we are able to reconstruct all previously known examples of fourth order differential CY-type operators of this type and provide new examples as well. A major part of the results of this article was developed in the first authors PhD-thesis \cite{Diss}. 
  
\vspace{2ex}

\textbf{Acknowledgements} We thank Duco van Straten for constant interest, support and suggestions concerning the content of this article. We are also indebted to Gert Almkvist for many inspiring conversations and attempts concerning the art of finding CY-type equations as well as his great effort to identify them.  

\section{Operations}
\subsection{Operations on tuples of matrices}

We fix some notations and conventions. 
In the sequel, we consider $G\subset \GL_n(\mathbb{C})$ to be an irreducible reductive linear algebraic group. 
A tuple of matrices $T=(T_1,\dots,T_{r+1})\in G^{r+1}$ is called an \textit{M-tuple} if 
\[\prod_{i=1}^{r+1}T_i=\mathbbm{1}_n\] holds. The matrix $T_i$ is called the \textit{i-th element} of the tuple. The \textit{rank} $\rk(T)$ of $T$ equals $n$. If $G$ is a symplectic or an orthogonal group, we call $T$ \textit{symplectic}, resp. \textit{orthogonal}, as well. 
Two M-tuples $T=(T_1,\dots,T_{r+1})$ and $T'=(T'_1,\dots,T'_{r+1})$ are \textit{equivalent} if there is an $H\in\GL_n(\mathbb{C})$ such that $H^{-1}T_iH=T'_i$ holds for all $1\leq i\leq r+1$.
We do not distinguish between M-tuples which are equivalent to each other. Moreover, we call two tuples of Jordan matrices to be \textit{similar}, if they coincide up to permutation of their matrices and tensor products with M-tuples of rank one.
For a matrix $A\in G$, we put
\[\gamma(A):=\rk(A-\mathbbm{1}_n)\] and
\[\delta_G(A):=\textrm{codim}(C_G(A)),\] the codimension w.r.t. $G$ of the centralizer $C_G(A)$ of $A$ in $G$.
Its Jordan form in $\GL_n(\mathbb{C})$ is denoted by $\J(A)$. For a given M-tuple $T$, we call
\[\J(T):=\left(\J(T_1),\dots,\J(T_{r+1})\right)\in\GL_n(\mathbb{C})^{r+1}\] the associated \textit{tuple of Jordan matrices}. Note that this tuple usually is not an M-tuple.
We study tuples of matrices of the following type

\begin{defin}
An arbitrary tuple of matrices $A=(A_1,\dots,A_{r+1})\in\GL_n(\mathbb{C})^{r+1}$ is called \textbf{of CY-type} if 
\begin{enumerate}
 \item $A$ is symplectic for $n$ even and orthogonal for $n$ odd.
\item all elements of $A$ are quasi-unipotent.
\item one of its elements $A_i$ is maximally unipotent, i.e. $\J(A_i)=J(n)$.
\end{enumerate}
 \end{defin}

Note that if $T$ is tuple of CY-type, its associated tuple of Jordan forms $\J(T)$ also is.

As CY-type operators are irreducible, we are mainly interested in \textit{irreducible} M-tuples, i.e. M-tuples whose elements generate an irreducible subgroup of $\GL_n(\mathbb{C})$. Irreducible M-tuples underly the following condition taken from \cite[Theorem 1]{Scott}.

 \begin{lemma}\label{ScottFormula}
 Each irreducible M-tuple $T=(T_1,\dots,T_{r+1})\in G^{r+1}$ of rank $n$ 
fulfills 
\[\sum_{i=1}^{r+1}\gamma(T_i)\geq 2n.\]
\end{lemma}

We also consider the rigidity index of tuples of matrices in $G$.

\begin{defin}
Consider a tuple of matrices $A\in G^{r+1}$. The positive integer
\[i_G(A):=\sum_{i=1}^{r+1}\delta_G(A_i) - 2(\dim G-\dim Z(G)),\] where $Z(G)$ denotes the center of $G$, is called the \textbf{index of rigidity} of $A$ in $G$.  
\end{defin}

If $T\in\GL_n(\mathbb{C})^{r+1}$ is an irreducible M-tuple, note that $i_{\GL_n(\mathbb{C})}(T)=0$ if and only if $T$ is linearly rigid, see e.g. \cite[Chapter 1]{katz96}.

For a given irreducible M-tuple $T\in G^{r+1}$, the main operations we are dealing with are \textit{tensor products} $T\otimes T'$ with other M-tuples $T'\in G^{r+1}$ and \textit{middle convolutions} $\MC_{\alpha}(T)$ where $\alpha\in\mathbb{C}^{*}$. Those operations are discussed in \cite{DRGalois}.
We just recall the related operations on the level of tuples of Jordan forms here.
Therefore, we denote in the sequel by $\alpha J(k)$ a Jordan block matrix of size $k$ with respect to the eigenvalue $\alpha$. The direct sum of two Jordan block matrices refers to a matrix splitting into those Jordan blocks. If a block $\alpha J(k)$ in such a decomposition appears with multiplicity $\nu$ we denote that by $\alpha J(k)^{\nu}$. All in all we may write
\[J=\bigoplus_{\alpha\in E(J)}\bigoplus_{k}\alpha J(k)^{\nu(\alpha,k)},\] where $E(J)$ is the set of eigenvalues of $J$.

We have a natural notion for the tensor product of two tuples $J=(J_1,\dots,J_{r+1})$ and $J'=(J'_1,\dots,J'_{r+1})$ of Jordan forms, namely
\[J\otimes J':=(\J(J_1\otimes J'_1),\dots,\J(J_{r+1}\otimes J'_{r+1})).\]
Very often, we take tensor products with rank one tuples of the form $(1,\dots,1,\alpha,1,\dots,1,\alpha^{-1},1,\dots)$ for some $\alpha\in\mathbb{C}^{*}$. Hence we denote by $K^j_i(\alpha)$ an M-tuple of rank one, which has $\alpha$ as its $i$-th element, $\alpha^{-1}$ as its $j$-th element and $1$ elsewhere. 

The middle convolution and the middle Hadamard product with special tuples of rank one on the level of Jordan forms are defined as follows:
\begin{defin}\label{JNF}
Consider an tuple of Jordan forms $J=(J_1,\dots,J_{r+1})\in \GL_n^{r+1}(\mathbb{C})$ and $\alpha\in\mathbb{C}^{*}$. Then we put
\[C_{\alpha}(J):=\sum_{i=1}^r\gamma(J_i)+\gamma\left(\alpha^{-1}J_{r+1}\right)-n,\  c_{\alpha,i}:=\gamma(J_i)+n-\gamma(\alpha J_i)\] for $1\leq i\leq r$ and $c_{\alpha,r+1}:=\gamma\left(\alpha^{-1}J_{r+1}\right)+n-\gamma(J_{r+1}).$ If $c_{\alpha,i}\leq C_{\alpha}(J)$ holds for all $1\leq i\leq r+1$, the tuple of Jordan matrices
\[\MC_{\alpha}(J):=(\MC_{\alpha}(J_1),\dots,\MC_{\alpha}(J_{r+1}))\]
 is defined via
\begin{align*}
 \MC_{\alpha}(J_i)=&
    \bigoplus_{\substack{\rho \in \mathbb{C} \setminus\{1,\alpha^{-1}\}}}\bigoplus_j\
         \alpha \rho J(j)^{v(\rho,j)}
    \bigoplus_{\substack{j\geq 2}} \alpha J(j-1)^{v(1,j)}
    \\&\quad\bigoplus J(j+1)^{v(\alpha^{-1},j)}\;
     \bigoplus J(1)^{C_{\alpha}(J)-c_{\alpha,i}(J)}        \\ 
\end{align*} 
 and
\begin{align*}
\MC_{\alpha}(J_{r+1})=&
    \bigoplus_{\rho \in \mathbb{C}\setminus\{1,\alpha\}}\bigoplus_j\ 
    \alpha^{-1} \rho J(j)^{v(r+1,\rho,j)}
    \bigoplus_{j\geq 2} J(j-1)^{v(\alpha,j)} \\
     &\quad\bigoplus \alpha^{-1} J(j+1)^{v(1,j)}
  \bigoplus \alpha^{-1} J(1)^{C_{\alpha}(J)-c_{\alpha,r+1}(J)}.
\end{align*}
Analogously, the middle Hadamard-product is defined via
\[\MH_{\alpha}(J)=\MC_{\alpha^{-1}}\left(J\otimes K^{r+1}_1(\alpha)\right)\]
\end{defin}
By \cite[Chapter 6]{katz96}, this operation is compatible with the middle convolution for irreducible M-tuples, i.e. we have
\[\J(\MC_{\alpha}(T))=\MC_{\alpha}(\J(T)).\]

The middle convolution $\MC_{\alpha}(T)$ with $\alpha\in\mathbb{C}^{*}\setminus\{1\}$ is an invertible operation which preserves irreducibility and rigidity of $T$, see \cite[Remark 3.1, Corollary 3.6 and Corollary 4.4]{DRGalois}.

\begin{prop}\label{ConstMDSPconv}
Consider an irreducible M-tuple $T=(T_1,\dots,T_{r+1})\in\GL_n(\mathbb{C})$ such that at least two matrices $T_i, T_j$ with $1\leq i,j\leq r$ are not the identity if $n=1$. Then for each $\alpha\in\mathbb{C}^{*}\setminus \{1\}$ we have that
\begin{enumerate}
\item $\MC_{\alpha}(T)$ is irreducible.
\item $\MC_{\alpha^{-1}}(\MC_{\alpha}(T))=T$.
\item $i_{\GL_n(\mathbb{C})}(T)=i_{\GL_n(\mathbb{C})}(\MC_{\alpha}(T))$.
\item The operation $\MC_{-1}$ turns symplectic into orthogonal and orthogonal into symplectic tuples.
\item If $T$ is orthogonal or symplectic, then
\[\MC_{\alpha\beta}\left(\MC_{(\alpha\beta)^{-1}}\left(T\otimes K^{r+1}_i(\alpha)\otimes K^{r+1}_j(\beta)\right)\otimes K^j_i(\alpha^{-1}\beta)\right)\otimes K^{r+1}_i\left(\beta^{-1}\right)\otimes K^{r+1}_j\left(\alpha^{-1}\right)\]
is orthogonal or symplectic.
\end{enumerate}
\end{prop}

The second and third statement also hold if we replace $T$ by $\J(T)$.

Finally, we also consider the \textit{square} of a tuple of Jordan forms $J=(J_1,\dots,J_{r+1})$ which is simply given by
\[J^2:=(\J(J_1^2), J_2, \dots,J_r, J_2,\dots,J_r, \J(J_{r+1}^2)).\]
In the geometric situation, this reflects the effect on the tuple of Jordan forms associated to a given local system under a degree two cover of the space it is defined on.

\subsection{Related operations on differential operators}

In this section, we review the translation of the constructions done on the level of monodromy tuples to the level of differential operators as it was done in \cite[Section 4, Section 5]{BoR} and \cite{Diss}.

We put as usual $\frac{d}{dz}$ to be the derivation on $\mathbb{C}[z]$ defined by $\frac{d}{dz}(z)=1$ and $\mathbb{C}[z,\partial]:=\mathbb{C}[z][\partial]$ to be the ring of differential operators with respect to $\frac{d}{dz}$. We will mainly use the so called \textit{logarithmic derivation} $z\frac{d}{dz}$ on $\mathbb{C}[z]$ and the ring of differential operators $\mathbb{C}[z,\vartheta]:=\mathbb{C}[z][\vartheta]$ with respect to $z\frac{d}{dz}$, which can naturally be regarded as a subring of $\mathbb{C}[z,\partial]$. We denote the degree of $L\in\mathbb{C}[z,\vartheta]$ with respect to $\vartheta$ by $\deg(L)$. Furthermore, we call $L=\sum_{i=0}^na_i\vartheta^i\in\mathbb{C}[z,\vartheta]$ to be \textit{reduced}, if all the $a_i$ are coprime in $\mathbb{C}[z]$. The polynomial $a_n$ is called the \textit{discriminant} of $L$. The set of roots of the discriminant together with the points $0$ and $\infty$ is the \textit{singular locus} $S$ of $L$. We can rearrange the coefficients of $L\in\mathbb{C}[z,\vartheta]$ and write $L=\sum_{i=0}^mz^iP_i$, with $P_i\in\mathbb{C}[\vartheta]$. The roots of $P_0$ are called the \textit{exponents} of $L$ at $z=0$. The exponents of $L$ at the other points $p\in\mathbb{C}$ and $p=\infty$ are defined in the same way via the transformation $z\mapsto z+p$, resp. $z\mapsto \frac{1}{z}$. As pointed out in the introduction, we can attach to each differential operator $L\in\mathbb{C}[z,\vartheta]$ with singular locus $S$ a local system $\mathbb{L}$ on $\mathbb{P}^1\setminus S$ and hence by further topological choices an M-tuple $T$ of matrices. In this situation, we usually call $T$ to be a \textit{monodromy tuple} of $L$.
 
There is a so called \textit{universal Picard-Vessiot ring} $\mathcal{F}$ of $\left(\mathbb{C}[z],z\frac{d}{dz}\right)$, i.e. a simple differential ring with field of constants $\mathbb{C}$ which contains all solutions of each differential operator $L\in\mathbb{C}[z,\vartheta]$. In particular, for each $L\in\mathbb{C}[z,\vartheta]$ the set $\Sol_L:=\{y\in\mathcal{F}\mid L(y)=0\}$ can be regarded as a $\deg(L)-$dimensional $\mathbb{C}$-vector space. Furthermore, a singularity $s$ of $L$ is \textit{regular} if each of its solutions $y\in\mathcal{F}$ can near $z=s$ be considered as element in $z^{\mu}\mathbb{C}\llbracket z\rrbracket[\ln(z)]$ where $\mu\in\mathbb{C}$. If all singularities $s_1,\dots,s_{r+1}$ of $L$ are regular, we call $L$ to be \textit{fuchsian}. In that case, we collect the local exponents $e_{1,i},\dots,e_{n,i}$ of $L$ at each of its singularities $s_i$ in the so called \textit{Riemann scheme}
\[\mathcal{R}\left(L\right)=\begin{Bmatrix}
s_1&\dots&s_{r+1}\\[0.1cm] \hline\\[-0.25cm]\begin{array}{c} e_{1,1}
\\\noalign{\medskip}\vdots
\\\noalign{\medskip}e_{n,1}\\\noalign{\medskip} \end{array}&\begin{array}{c} \dots
\\\noalign{\medskip}\vdots
\\\noalign{\medskip}\dots\\\noalign{\medskip} \end{array}
&\begin{array}{c} e_{1,r+1}
\\\noalign{\medskip}\vdots
\\\noalign{\medskip}e_{n,r+1}\\\noalign{\medskip} \end{array}
\end{Bmatrix}\]

Given two reduced differential operators $L_1,L_2\in \mathbb{C}[z,\vartheta]$, their \textit{tensor product} $L_1\otimes L_2\in \mathbb{C}[z,\vartheta]$ is the reduced operator of minimal degree, whose solution space is spanned by the set $\{y_1y_2\in\mathcal{F}\mid L_1(y_1)=L_2(y_2)=0\}$. Furthermore, if $g$ is algebraic over $\mathbb{C}(z)$ and spans the solution space of $L_2$, we write \[L_1\otimes L_2=:L_1\otimes g.\] If $T_i$ is the monodromy tuple associated to $L_i$, the monodromy tuple associated to $L_1\otimes L_2$ is a sub tuple of $T_1\otimes T_2$, i.e. this tuple can be achieved by restricting the elements $T_1\otimes T_2$ to a subspace of $\mathbb{C}^n$ which is set-wise invariant under all of them. In particular, we have $\deg(L_1\otimes L_2)\leq \deg(L_1)\deg(L_2)$.

To discuss the middle Hadamard product, we introduce for each $a\in\mathbb{Q}$ the operator \[I_a:=\vartheta-z(\vartheta+a)\in\mathbb{C}[z,\vartheta].\]
Moreover, writing $L=\sum_{i=0}^mz^iP_i\in\mathbb{C}[z,\vartheta]$, we put
\[\mathcal{H}_a(L):=\sum_{i=0}^mz^i\prod_{j=0}^{i-1}(\vartheta+a+j)\prod_{k=0}^{m-i-1}(\vartheta-k)P_i.\]
If $L$ is fuchsian, irreducible and admits no local exponent in $\mathbb{Z}_{<0}$ at each $p\in\mathbb{P}^1\setminus\{0\}$ and no local exponent in $1-a+\mathbb{Z}_{<0}$ at $z=0$
then the \textit{middle Hadamard product} $L\star_H I_a$ of $L$ and $I_a$ is given by the unique irreducible right factor of $\mathcal{H}_a(L)$ of degree
\[\deg(L\star_H I_a)=\sum_{i=2}^{r+1}\gamma(T_i)-n+\gamma(\exp(2\pi ia)T_{1}),\] where $T=(T_1,\dots,T_{r+1})$ denotes a monodromy tuple associated to $L$ such that $T_1$ is the local monodromy at $z=0$ and $T_{r+1}$ the local monodromy at $z=\infty$.

Furthermore, we investigate the natural action of rational functions for fuchsian differential operators $\mathbb{C}[z,\vartheta]$. As $\mathbb{C}[z,\vartheta]$ is a differential subring of $\mathbb{C}(z)[\partial]$ and for each fuchsian $L\in\mathbb{C}(z)[\partial]$ there is an $n\in\mathbb{N}$ such that $z^nL\in\mathbb{C}[z,\vartheta]$, we rather describe those actions on $\mathbb{C}(z)[\partial]$. More precisely, for each $\varphi\in\mathbb{C}(z)\setminus\mathbb{C}$ we define the action of $\varphi^{*}$ on $\mathbb{C}(z)[\partial]$ via
\[f\mapsto f\circ\varphi,\ \partial\mapsto \frac{1}{\varphi(z)'}\partial\] for each $f\in\mathbb{C}(z)$.

We will also use some algebraic transformations, namely the inverse of $z\mapsto z^n$ for $n\in\mathbb{N}$, if they are defined. To be more precise, we define the differential subring $\mathbb{C}\left[z^n,\vartheta\right]$ of $\mathbb{C}[z,\vartheta]$ to consist of
\[\mathbb{C}\left[z^n,\vartheta\right]=\left\{\sum_{i=0}^mz^{ni}P_i \mid P_i\in\mathbb{C}[\vartheta]\right\}\]
and set \[\left(z^{\frac{1}{n}}\right)^{*}\colon\mathbb{C}[z^n,\vartheta]\to\mathbb{C}[z,\vartheta],\ z^n\mapsto z,\ \vartheta\mapsto n\vartheta.\]

One can check locally, if an operator lies in $\mathbb{C}[z^n,\vartheta]$ using the following

\begin{lemma}\label{root}
Consider an irreducible, fuchsian differential operator $L\in\mathbb{C}[z,\vartheta]$ and $n\in\mathbb{N}$. Then $L\in \mathbb{C}[z^n,\vartheta]$ if and only if $L$ has a local solution $f\in z^{\mu}\mathbb{C}\llbracket z^n\rrbracket$ near $z=0$. 
\end{lemma}

\begin{proof}
We write $L=\sum_{i=0}^mz^iP_i$. Let $f=\sum_{\nu=\mu}^{\infty}A(\nu)z^\nu:=z^{\mu}\sum_{k=0}^{\infty}B_kz^k$ a local solution of $L$ at $z=0$. Then the $A(\nu)$ fulfill the recurrence equation
\[\sum_{i=0}^mP_i(\nu-i)A(\nu-i)=0,\] where we put $A(\nu)=0$ if $\nu<\mu$. Assume now that $A(\nu)=0$ if $\nu\neq kn\mu$ for any $k\in\mathbb{N}$. Writing $m=pn+q$, $p,q\in\mathbb{N}_0, q<n$, we see that
\[\sum_{k=0}^pP_{nk}(\nu-nk)A(\nu-nk)=0.\] Therefore, $f$ is a solution of $\sum_{k=0}^pz^{nk}P_{nk}\in\mathbb{C}[z^n,\vartheta]$. As $L$ is an irreducible right factor of this operator and has the same degree, they have to coincide. To prove the other direction of the statement let $L\in \mathbb{C}[z^n,\vartheta]$ and assume that $L$ has no local solution in $z^{\mu}\mathbb{C}\llbracket z^n\rrbracket$. Without loss of generality we can assume that $\mu+r$ is not an exponent of $L$ for each $r\in\mathbb{N}$. Let $\nu_0$ be minimal such that $\nu_0$ is no integer multiple of $\mu$ and $A(\nu_0)\not=0$. By the recurrence equation stated above, we get $P_0(\nu_0)A(\nu_0)=0$ in contradiction to our assumptions.
\end{proof}

The previous lemma justifies the following construction, which we call the \textit{Delta construction}:

Consider an irreducible, fuchsian differential operator $R\in\mathbb{C}[z,\vartheta]$ and a local solution $f=z^{\mu}\sum_{m=0}^{\infty}A_mz^m$ of $R$ near $z=0$. Then $z^{2\mu}\sum_{m=0}^{\infty}\sum_{k=0}^m(-1)^kA_kA_{m-k} z^m$ is a solution of $R\otimes (-z)^{*}(R)$, which lies in $z^{2\mu}\mathbb{C}\llbracket z^2\rrbracket$. Thus we can apply the map $z^{1/2}$ to this operator and end up with
\[\Delta(R):=\left(z^{1/2}\right)^{*}\left(R\otimes \left(-z^{*}\right)R\right)\in\mathbb{C}[z,\vartheta].\]

\section{CY-type tuples of rank four with symplectic rigidity index two}

We first give a list of CY-type tuples of Jordan-forms $J\in\Sp_4(\mathbb{C})^{r+1}$ with $i_{\Sp_4(\mathbb{C})}(J)=2$ and discuss possible constructions of related M-tuples thereafter. To simplify notation, we denote by $[\alpha]$ the Jordan block matrix $\alpha J(1)\oplus \alpha^{-1} J(1)$ if $\alpha\neq \alpha^{-1}$ and $\alpha J(2)$ else. The notation $[[\alpha]\oplus [\beta]]$ is used similarly. The following table collects for each matrix $A\in\Sp_4(\mathbb{C})$ the codimension of its centralizer $\delta_{\Sp_4(\mathbb{C})}(A)$ and its Jordan form of its image under the exceptional isomorphism $\Sp_4(\mathbb{C})\cong\SO_5(\mathbb{C})$. 

\begin{longtable}{c c c}\label{DeltaSPfour}
$\delta_{\Sp_4(\mathbb{C})}$&\textbf{Jordan form in} $\Sp_4(\mathbb{C})$&\textbf{Jordan form in} $\SO_5(\mathbb{C})$\\[2mm]
\multirow{2}{*}{$4$}&$\pm(J(2)\oplus J(1)^2)$&$J(2)\oplus J(2)\oplus J(1)$\\[2mm]& $-J(1)^2\oplus J(1)^2$& $-J(1)^4\oplus J(1)$\\[5mm]
\multirow{2}{*}{$6$}&$[\alpha]^2$&$\left[\left[\alpha^2\right]\oplus J(1)\right]\oplus J(1)^2$\\[2mm]&$\pm([\alpha]\oplus J(1)^2)_{\alpha\not=1}$&$[\alpha]^2\oplus J(1)$\\[5mm]
\multirow{2}{*}{$8$}&$[[\alpha]\oplus[\beta]]_{\alpha\neq \beta^{\pm 1}}$&$[[\alpha\beta]\oplus [\alpha\beta^{-1}]\oplus J(1)]$\\[2mm]&
$[[\alpha]\oplus [\alpha]]$&$J(3)\oplus \alpha^2\oplus \alpha^{-2}$\\[0.5mm] 
\caption{Jordan forms in $\Sp_4(\mathbb{C})$ and $\SO_5(\mathbb{C})$.}
\end{longtable}

\begin{prop}\label{Classtup}
Consider an M-tuple $T\in \Sp_4(\mathbb{C})^{r+1}$ with $i_{\Sp_4(\mathbb{C})}(T)$ which is induced by an operator CY-type. Then we have $r\in\{2,3\}$ and $\J(T)$ is similar to one of the following tuples
\begin{enumerate}
\item $N_1(\alpha,\beta,\gamma)=(J(4),[\alpha]^2,[[\beta]\oplus[\gamma]])$.
\item $N_2(\alpha,\beta,\gamma)=(J(4),[\alpha]\oplus J(1)^2,[[\beta]\oplus[\gamma]])$.
\item $M_1(\alpha)=\left(J(4), J(2)\oplus J(1)^2, J(2)\oplus J(1)^2, [\alpha]^2\right)$ for $\alpha\neq 1$.
\item $M_2(\alpha)=\left(J(4), J(2)\oplus J(1)^2, J(2)\oplus J(1)^2, [\alpha]\oplus -J(1)^2\right)$ for $\alpha\neq 1$.
\item $M_3(\alpha)=\left(J(4), J(2)\oplus J(1)^2, J(1)^2\oplus -J(1)^2, [\alpha]\oplus J(1)^2\right)$ for $\alpha\neq 1$.
\item $M_4(\alpha)=\left(J(4), J(2)\oplus J(1)^2, J(1)^2\oplus -J(1)^2, [\alpha]^2\right)$.
\item $M_5(\alpha)=\left(J(4), J(1)^2\oplus -J(1)^2, J(1)^2\oplus -J(1)^2, [\alpha]\oplus -J(1)^2\right)$.
\end{enumerate}
\end{prop}

\begin{proof}
Throughout the proof, write $J:=\J(T)$ and $G=\Sp_4(\mathbb{C})$. We assume without loss of generality that $J_1=J(4)$ and put the other matrices into any order. As $i_{G}(J)=2$ and $\delta_{G}(J(4))=8$, we have \[\sum_{i=2}^{r+1}\delta_{G}(J_i)-14=0.\] By Table \ref{DeltaSPfour}, this yields $r\leq 3$. 
The possibilities for $r=2$ give the families $N_1(\alpha,\beta,\gamma)$ and $N_2(\alpha,\beta,\gamma)$.
If $r=3$, we have
$\delta_{G}(J_2)=\delta_{G}(J_3)=4$ and $\delta_{G}(J_4)=6$. In particular, we get that
\[J_2,J_3\in\{\pm\left(J(2)\oplus J(1)^2\right), -J(1)^2\oplus J(1)^2\}\]
and \[J_4\in\left\{[\alpha]^2, \pm\left(\left[\alpha\right]_{\alpha\not=1}\oplus J(1)^2\right)\right\}.\]  
If $J_2=J_3=J(2)\oplus J(1)^2$, the cases for which $\gamma(J_4)=2$ are excluded by Lemma \ref{ScottFormula}. This leaves the possibilities $J_4=[\alpha]^2$ and $J_4=[\alpha]\oplus- J(1)^2$, where $\alpha\neq 1$.
 
If $J_2=J(1)^2\oplus -J(1)^2$ and $J_3=J(2)\oplus J(1)^2$, the cases for which $\gamma(J_4)=1$ are excluded.

If $J_2=J_3=J(1)^2\oplus -J(1)^2$, we get that $\gamma\left(\bigwedge^2J_1\right)+\gamma\left(-\bigwedge^2J_2\right)+\gamma\left(-\bigwedge^2J_3\right)=6$ as we are looking for solutions in $G$. Therefore, the possibilities for which $\gamma\left(\bigwedge^2J_4\right)\leq 3$ are ruled out by Lemma \ref{ScottFormula}. This discussion leaves the possible tuples $M_1(\alpha)-M_5(\alpha)$.
\end{proof}

\subsection{Constructions for tuples with three elements}
We do not have a general construction for tuples in families $N_1(\alpha,\beta,\gamma)$ and $N_2(\alpha,\beta,\gamma)$ stated in Proposition \ref{Classtup}.
Nevertheless, we point out possible constructions for special parameters $\alpha,\beta,\gamma\in\mathbb{C}^{*}$. 

Consider $\alpha,\beta,\gamma\in \mathbb{C}^{*}$ and a tuple of CY-type $T$ such that $\J(T)=N_1(\alpha,\beta,\gamma)$.
Then setting
\[\tilde{T}=\MC_{\alpha\beta}\left(\MC_{(\alpha\beta)^{-1}}\left(T\otimes K^3_2(\alpha)\right)\otimes K^3_2(\alpha^{-1}\beta)\right)\otimes K_2^3(\beta)\]
we have $\J(\tilde{T})=N_1(\beta,\alpha,\gamma)$. As $\tilde{T}$ has an element of maximally unipotent monodromy, it is symplectic by Proposition \ref{ConstMDSPconv} and hence of CY-type.

Starting with an operator which is a symmetric cube of a second order one and translating this operation to the level of differential operators, we hence obtain the following result 

\begin{prop}\label{Dreieins}
Let $a\in\mathbb{Q}\setminus \mathbb{Z}$ and 
\begin{align*}
S_a&=\Sym^3(I_{\frac{1}{4}+a}\star I_{\frac{1}{4}-a})\\&=256\,{\vartheta}^{4}+z \left( -512\,{\vartheta}^{4}-768\,{\vartheta}^{
3}-672\,{\vartheta}^{2}+2560\,{\vartheta}^{2}{a}^{2}+2560\,\vartheta\,
{a}^{2}-288\,\vartheta+768\,{a}^{2}-48 \right)\\& +{z}^{2} \left( 4\,
\vartheta+3+12\,a \right)  \left( 4\,\vartheta+3-12\,a \right) 
 \left( 4\,\vartheta+3+4\,a \right)  \left( 4\,\vartheta+3-4\,a
 \right) 
\end{align*}
Put
\begin{align*}
P_{a,1}&:=\left(\left(S_a\otimes z^{-\frac{1}{4}-a}\right)\star I_{\frac{5}{4}+a}\otimes z^{\frac{1}{4}+a} (1-z)^{a-\frac{1}{4}}\right)\star I_{\frac{3}{4}-a}\\& =16\,{\vartheta}^{4}-z\left(32\vartheta^4+64(1-a)\vartheta^3+\left(63-96a-112a^2\right)\vartheta^2\right)\\&\quad-z\left(\left(31-64a-112a^2\right)\vartheta+6-16a-32a^2\right)\\&\quad+{z}^{2} \left( 4\,\vartheta+5-4\,a
 \right)  \left( \vartheta+1-4\,a \right)  \left( \vartheta+1+2\,a
 \right)  \left( 4\,\vartheta+3-4\,a \right)
\end{align*}
and
\begin{align*}P_{a,2}&:=\left(\left(S_a\otimes z^{-\frac{1}{4}-3a}\right)\star I_{\frac{5}{4}+3a}\otimes z^{3a+\frac{1}{4}}(1-z)^{3a-\frac{1}{4}}\right)\star I_{\frac{3}{4}-3a}\\&=16\,{\vartheta}^{4}+z \left( -32\,{\vartheta}^{4}-64\,{\vartheta}^{3}+
192\,{\vartheta}^{3}a-63\,{\vartheta}^{2}-272\,{\vartheta}^{2}{a}^{2}+
288\,{\vartheta}^{2}a\right)\\&\quad+192z\left(\,\vartheta\,a-272\,\vartheta\,{a}^{2}-31\,
\vartheta-6+48\,a-96\,{a}^{2} \right)\\&\quad +{z}^{2} \left( 4\,\vartheta+3-
12\,a \right)  \left( \vartheta+1-2\,a \right)  \left( \vartheta+1-4\,
a \right)  \left( 4\,\vartheta+5-12\,a \right) 
\end{align*}
Then the tuple of Jordan forms associated to each monodromy tuple of $P_{a,1}\otimes (1-z)^{-a}$ is similar to $N_1(\exp(2\pi ia), i, \exp(6\pi ia))$ and the tuple of Jordan forms associated to each monodromy tuple of $P_{a,2}\otimes (1-z)^{-3a}$ is similar to $N_1(\exp(6\pi ia), i, \exp(2\pi ia))$. 
\end{prop}

In Section \ref{Opthree}, we indicate for which values of $a$ we get CY-type operators. Note, that although the operator $S_a$ we started with is a symmetric cube, the operator $P_{a,1}$ is usually not.

For an M-tuple $T$ of CY-type whose tuple of Jordan forms reads $N_2(\alpha,\beta,-1)$ a direct computation shows
that \[\J(\MC_{-1}(T))^2=J\otimes J\] for $J=(J(2),[\delta], [\delta], [\epsilon])$ with $-\delta^2=\alpha$ and $\epsilon^2=\beta$. 
We try to invert this construction on the level of differential operators. Therefore, we should start with an irreducible operator $R$ of order two whose non-apparent singularities are $S=\{0,-1,1,\infty\}$ and such that the tuple of Jordan forms of a monodromy tuple coincides with $J$. Moreover, we require $R\not\in\mathbb{C}[z^2,\vartheta]$ to get that $\Delta(R)$ is irreducible of degree four.
Then the irreducible right factors of $\Delta(R)\star_H I_{\frac{1}{2}}$ are good candidates for CY-type operators of order four.
So far, we only found two special differential operators for which this construction works, see Section \ref{Opthree}.

\subsection{Constructions for tuples with four elements}

Unless the tuples $N_1(\alpha,\beta,\gamma)$ and $N_2(\alpha,\beta,\gamma)$, we can construct for each $M_i(\alpha)$ a family of inducing differential operators containing examples of CY-type. It turns out that the starting point for those general constructions are special irreducible operators of order two with four singularities. We introduce the following terminology:

\begin{defin}
For a fuchsian differential operator $L\in\mathbb{C}[z,\vartheta]$ of degree two which has rational exponents $e_{1,s}\leq e_{2,s}$ at each singularity $s\in S$, we call \[\lambda_s:=e_{2,s}-e_{1,s}\] the \textit{signature} of the singularity $s$. With respect to an order on $S$, the tuple of signatures of all points in $S$ is denoted by $\sign(L)$ and called the \textbf{signature} of $L$.
\end{defin}

In the sequel, we consider fuchsian differential operators $L$ for which the smallest exponent at $z=s$ is zero at each $s\in \mathbb{C}$. Then, by the classical Fuchs relation \[e_{1,\infty}+e_{2,\infty}=|S|-2-\sum_{s\in S\setminus\{\infty\}}\lambda_s,\] all remaining exponents of $L$ are determined by its $\sign(L)$.
Moreover, as we only want to consider operators for which the exponents at each of their singularities fix the Jordan form of the corresponding local monodromy, we assume that all entries in $\sign(L)$ are strictly smaller than one. A class of such operators are the following ones:

\begin{defin}
For $h:=(t,u,v,w,c,s_1,s_2)\in\left(\mathbb{C}\cap [0,1)\right)^4\times \mathbb{C}\times \mathbb{C}\setminus\{0\}\times \mathbb{C}\setminus\{0,s_1\}$, we call 
\begin{align*}R(h)&:=4\,\vartheta\,s_{{1}}s_{{2}} \left( \vartheta-t \right)-4z(\vartheta^2(s_1+s_2)+\vartheta(s_1(1-v)+s_2(1-u)-t(s_1+s_2))+c)\\&+{z}^{2}
 \left( 2\,\vartheta+2-t-u-v+w \right)  \left( 2\,\vartheta+2-t
-u-v-w \right). 
\end{align*}
the associated \textbf{Heun operator}. 
\end{defin}

The Riemann-scheme of a Heun operator reads
\[\mathcal{R}\left(R(h)\right)=\begin{Bmatrix}
0&s_1&s_2&\infty\\[0.1cm] \hline\\[-0.25cm]\begin{array}{c} 0
\\\noalign{\medskip}t
\\\noalign{\medskip} \end{array}&\begin{array}{c} 0
\\\noalign{\medskip}u
\\\noalign{\medskip} \end{array}&\begin{array}{c} 0
\\\noalign{\medskip}v
\\\noalign{\medskip}\end{array}&\begin{array}{c}1-\frac{1}{2}(t+u+v+w) 
\\\noalign{\medskip}1-\frac{1}{2}(t+u+v-w)
\\\noalign{\medskip}\end{array}
\end{Bmatrix}\] and its signature is given by $\sign(R(h))=(t,u,v,w)$.
In the sequel, we will just study Heun operators $R(h)\in\mathbb{Q}[z,\vartheta]$, meaning that either $s_1,s_2\in\mathbb{Q}$ or $u=v$ and there is an irreducible polynomial $p\in\mathbb{Q}[X]$ of degree two such that $p(s_1)=p(s_2)=0$.

We state constructions for each of the cases $M_1(\alpha)-M_5(\alpha)$ using Heun operators. This leads to three-parameter families of operators. Operators of CY-type inside those families are stated in Section \ref{Opfour}.

Assume in the sequel that $T$ is an M-tuple of CY-type. 
If \[\J(T)=M_1(\alpha)=\left(J(4), J(2)\oplus J(1)^2, J(2)\oplus J(1)^2, [\alpha]^2\right)\] for $\alpha\neq 1$ we find that
\begin{align*}\MC_{\alpha}\left(\MC_{\alpha^{-1}}(M_1(\alpha))\otimes K_1^4(\alpha)\right)\otimes K_4^1(\alpha)=(J(2), J(2), J(2), J(2)).\end{align*} As the tuple of Jordan forms associated to the monodromy tuple of $R(0,0,0,0,s_1,s_2,c)$ is precisely $(J(2), J(2), J(2), J(2))$, translating this construction to the level of differential operators yields

\begin{prop}
For each $a\in\mathbb{Q}\setminus\mathbb{Z}$ the tuple of Jordan forms associated to each monodromy tuple of
\begin{align*}
Q_1(c,s_1,s_2,a)&=R(0,0,0,0,s_1,s_2,c)\star_H I_a\star_H I_{1-a}\\&={\vartheta}^{4}s_{{1}}s_{{2}}-z \left( \vartheta+a \right)  \left( 
\vartheta+1-a \right)  \left( {\vartheta}^{2}s_{{1}}+{\vartheta}^{2}s_
{{2}}+\vartheta\,s_{{1}}+\vartheta\,s_{{2}}+c \right)\\& +{z}^{2} \left( 
\vartheta+2-a \right)  \left( \vartheta+1-a \right)  \left( \vartheta+
1+a \right)  \left( \vartheta+a \right) 
\end{align*} is similar to $M_1(\exp(2\pi ia))$. 
\end{prop}

If \[\J(T)=M_2(\alpha)=\left(J(4), J(2)\oplus J(1)^2, J(2)\oplus J(1)^2, [\alpha]\oplus -J(1)^2\right)\] for $\alpha\neq 1$, we find that
\[\MC_{-1}(M_2(\alpha))\otimes K^4_1(-1)=\Sym^2((J(2), J(1)\oplus -J(1), J(1)\oplus -J(1), [\beta]))\] for $\beta^2=\alpha$. As the tuple of Jordan forms associated to the monodromy tuple of $R\left(0,\frac{1}{2}, \frac{1}{2}, \lambda, s_1, s_2, c\right)$ is similar to $(J(2), J(1)\oplus -J(1), J(1)\oplus -J(1), [\exp(2\pi i\lambda)])$, we get

\begin{prop}
For each $\lambda\in\mathbb{Q}\setminus\mathbb{Z}$ the tuple of Jordan forms associated to each monodromy tuple of
\begin{align*}
Q_2(c,s_1,s_2,\lambda)&= \Sym^2\left(R\left(0,\frac{1}{2}, \frac{1}{2}, \lambda, s_1, s_2, c\right)\right)\star_H I_{\frac{1}{2}}\\&= 4\,{\vartheta}^{4}s_{{1}}s_{{2}}-z \left( 2\,\vartheta+1 \right) ^{2}
 \left((s_{{1}}+s_{{2}})(\vartheta^2+\vartheta)+4\,c \right)\\& +{z}^{2} \left( 2\,\vartheta+3
 \right)  \left( 2\,\vartheta+1 \right)  \left( \vartheta+1+\lambda
 \right)  \left( \vartheta+1-\lambda \right) 
\end{align*}
is similar to $M_2(\exp(2\pi i\lambda))$. 
\end{prop}

If \[\J(T)=M_3(\alpha)=\left(J(4), J(2)\oplus J(1)^2, J(1)^2\oplus -J(1)^2, [\alpha]\oplus J(1)^2\right)\] for $\alpha\neq 1$, we find
\begin{align*}J'&=\MC_{-1}(M_3(\alpha)\otimes K^4_1(-1))\otimes K_1^4(-1)\\&=(J(3)\oplus -J(1), -J(1)\oplus J(1)^3, J(2)^2, [[-\alpha]\oplus J(1)]\oplus J(1)).
\end{align*}
The latter tuple is similar to
$J''=(J(3)\oplus -J(1), [[-\alpha]\oplus J(1)]\oplus J(1),J(2)^2, -J(1)\oplus J(1)^3)$ for which we find that
\begin{align*}(J'')^2=\left(J(2),[\beta], J(1)^2,[\beta],J(2), J(1)^2\right)\otimes \left(J(2), [\beta], J(2), [\beta], J(1)^2, J(1)^2\right)\end{align*} with  $\beta^2=-\alpha$.
To mimic the inverse construction on the level of differential operators, we start with the Heun operator $R(0,\lambda, \lambda,0,s_1,s_2,c)$.

Consider 
\[P_1=\left(\frac{2zs_1s_2}{z(s_1+s_2)+s_2-s_1}\right)^{*}R(0,\lambda, \lambda,0,s_1,s_2,c)\otimes (z^2-1)^{-\lambda/2}((s_1+s_2)z-(s_1-s_2))^{\lambda-1},\] whose Riemann scheme reads
\[\mathcal{R}(P_1)=\begin{Bmatrix}
0&-1&1&(s_1-s_2)/(s_1+s_2)&\infty\\[0.1cm] \hline\\[-0.25cm]\begin{array}{c} 0
\\\noalign{\medskip}0
\\\noalign{\medskip} \end{array}&\begin{array}{c} -\frac{\lambda}{2}
\\\noalign{\medskip}\frac{\lambda}{2}
\\\noalign{\medskip} \end{array}&\begin{array}{c} -\frac{\lambda}{2}
\\\noalign{\medskip}\frac{\lambda}{2}
\\\noalign{\medskip}\end{array}&\begin{array}{c} 0
\\\noalign{\medskip}0
\\\noalign{\medskip}\end{array}&\begin{array}{c} 1
\\\noalign{\medskip}2
\\\noalign{\medskip}\end{array}
\end{Bmatrix}.\]
The tuple of Jordan forms associated to each monodromy tuple of $P_1$ is similar to $(J(2),[\exp(\pi i\lambda)], [\exp(\pi i\lambda)], J(2), J(1)^2)$. For general values $c\in\mathbb{C}$, a direct computation shows, that the Riemann scheme of $\Delta(P_1)$ reads
\[\mathcal{R}(\Delta(P_1))=\begin{Bmatrix}
0&1&(s-1)^2/(s+1)^2&\tilde{s}&\infty\\[0.1cm] \hline\\[-0.25cm]
\begin{array}{c} 0
\\\noalign{\medskip}0
\\\noalign{\medskip}0
\\\noalign{\medskip}\frac{1}{2}
\\\noalign{\medskip} \end{array}&\begin{array}{c}0
\\\noalign{\medskip}1
\\\noalign{\medskip}-\lambda
\\\noalign{\medskip}\lambda
\\\noalign{\medskip}\end{array}&\begin{array}{c}0
\\\noalign{\medskip}0
\\\noalign{\medskip}1
\\\noalign{\medskip}1
\\\noalign{\medskip}\end{array}
&\begin{array}{c}0
\\\noalign{\medskip}1
\\\noalign{\medskip}2
\\\noalign{\medskip}4
\\\noalign{\medskip}\end{array}
&\begin{array}{c}1
\\\noalign{\medskip}\frac{3}{2}
\\\noalign{\medskip}2
\\\noalign{\medskip}3
\\\noalign{\medskip}\end{array}
\end{Bmatrix}
,\]
where $\tilde{s}\in\mathbb{C}$ is an apparent singularity. Next, we check that $\Delta(P_1)\star_H I_{\frac{1}{2}}=LP_2$, where $L$ is the unique monic operator of degree one whose Riemann-scheme reads
\[\mathcal{R}(L)=\begin{Bmatrix}
0&1&(s-1)^2/(s+1)^2&\tilde{s}&\infty\\[0.1cm] \hline\\[-0.25cm]
\begin{array}{c} -3
\\\noalign{\medskip}\end{array}&\begin{array}{c} -2
\\\noalign{\medskip}\end{array}&\begin{array}{c} -3
\\\noalign{\medskip}\end{array}&\begin{array}{c} \frac{1}{2}
\\\noalign{\medskip}\end{array}&\begin{array}{c} \frac{15}{2}
\\\noalign{\medskip}\end{array}\end{Bmatrix}.\]

The tuple of Jordan forms associated to each monodromy tuple of $P_2$ is similar to $M_3(\exp(2\pi i(\lambda+1/2))$. To decrease the number of terms of the resulting differential operator, we put $Q_3(s_1,s_2,\lambda,c):=\left(\frac{z}{z-1}\right)^{*}(P_2)\otimes ((1-2x)(4s_1s_2x+(s_1-s_2)^2))^{-1/2}$. Carrying out all computations explicitly yields

\begin{prop}
For each $\lambda\in\mathbb{Q}\setminus\mathbb{Z}$ the tuple of Jordan forms associated to each monodromy tuple of 
\begin{align*}
Q_3(s_1,s_2,\lambda,c)&={\vartheta}^{4} \left( s_{{1}}-s_{{2}} \right) ^{4}-z \left( s_{{1}}-s
_{{2}} \right) ^{2} ((s^2_1-10s_1s_2+s^2_2)(\vartheta^4+2\vartheta^3))\\&-z \left( s_{{1}}-s
_{{2}} \right) ^{2}\vartheta^2\left((s_1^2+s_2^2)(1+\lambda-\lambda^2)+2s_1s_2(\lambda^2+\lambda-12)+2c(s_1+s_2)\right)
\\&-z \left( s_{{1}}-s
_{{2}} \right) ^{2}\vartheta\left((s_1^2+s_2^2)\lambda(1-\lambda)+2s_1s_2(\lambda^2+\lambda-7)+2c(s_1+s_2)\right)
\\&-z \left( s_{{1}}-s
_{{2}} \right) ^{2}(s_1s_2(\lambda^2-3)+c\lambda(s_1+s_2)+c^2)
\\&-4z^2s_1s_2(\vartheta+1)^2\left(2(s^2_1-4s_1s_2+s^2_2)(\vartheta^2+2\vartheta)\right)\\&-4z^2s_1s_2(\vartheta+1)^2\left((s_1^2+s_2^2)(3+\lambda-2\lambda^2)+s_1s_2(4\lambda^2+2\lambda-13)+2c(s_1+s_2)\right)
\\&-4z^3s_1^2s_2^2(\vartheta+1)(\vartheta+2)(2\vartheta+3+2\lambda)(2\vartheta+3-2\lambda)
\end{align*}
is similar to $M_3(\exp(2\pi i(\lambda+1/2)))$. 
\end{prop}

If \[\J(T)=M_4(\alpha)=\left(J(4), J(2)\oplus J(1)^2, J(1)^2\oplus -J(1)^2, [\alpha]^2\right)\] we have
\begin{align*}
&\MC_{-1}\left(\MC_{\alpha}\left(\MC_{\alpha^{-1}}(M_4(\alpha))\otimes K^4_1(\alpha) \right)\otimes K_4^1(\alpha)\otimes K^4_3(-1)\right)\otimes K^4_1(-1)\\&=(J(1)\oplus -J(1)\oplus \alpha J(1)\oplus \alpha^{-1} J(1), J(2)^2, -J(1)\oplus J(1)^3, -J(1)^2\oplus J(1)^2).
\end{align*}
This tuple is similar to $\tilde{J}=(J(1)\oplus -J(1)\oplus \alpha J(1)\oplus \alpha^{-1} J(1), J(2)^2, -J(1)^2\oplus J(1)^2), -J(1)\oplus J(1)^3)$ for which we find that
\begin{align*}
\tilde{J}^2=\left([\beta],J(2), [i], J(1)^2, [i], J(1)^2\right)\otimes\left([\beta], J(1)^2, [i], J(2)\otimes J(1)^2, [i], J(1)^2\right),
\end{align*}
where $\beta^2=\alpha$.
To mimic the inverse construction on the level of differential operators, we start with the Heun operator $R(0,\frac{1}{2}, \frac{1}{2},\lambda,s_1,s_2,c)$.
Consider the operator 
\[P_1:=\left(\frac{(s_1+s_2)z+s_1-s_2}{2z}\right)^{*}R(0,\frac{1}{2}, \frac{1}{2},\lambda,s_1,s_2,c)\otimes z^{-1/2}(z^2-1)^{1/4}\]
whose Riemann scheme reads
\[\mathcal{R}(P_1)=\begin{Bmatrix}
0&-1&1&(s_1-s_2)/(s_1+s_2)&\infty\\[0.1cm] \hline\\[-0.25cm]\begin{array}{c} -\frac{\lambda}{2}
\\\noalign{\medskip}\frac{\lambda}{2}
\\\noalign{\medskip} \end{array}&\begin{array}{c} \frac{1}{4}
\\\noalign{\medskip}\frac{3}{4}
\\\noalign{\medskip} \end{array}&\begin{array}{c} \frac{1}{4}
\\\noalign{\medskip}\frac{3}{4}
\\\noalign{\medskip}\end{array}&\begin{array}{c} 0
\\\noalign{\medskip}0
\\\noalign{\medskip}\end{array}&\begin{array}{c} 0
\\\noalign{\medskip}1
\\\noalign{\medskip}\end{array}
\end{Bmatrix}.\]

Moreover, the tuple of Jordan forms associated to each monodromy tuple of $P_1$ is similar to $\left([\exp(\pi i\lambda)], [i], [i], J(2), J(1)^2\right)$.
For general values $c\in\mathbb{C}$, a direct computation shows that the operator $\Delta(P_1)$ has Riemann scheme
\[\mathcal{R}(\Delta(P_1))=\begin{Bmatrix}
0&1&(s-1)^2/(s+1)^2&\tilde{s}&\infty\\[0.1cm] \hline\\[-0.25cm]
\begin{array}{c} 0
\\\noalign{\medskip}\frac{1}{2}
\\\noalign{\medskip}-\frac{\lambda}{2}
\\\noalign{\medskip}\frac{\lambda}{2}
\\\noalign{\medskip} \end{array}&\begin{array}{c}\frac{1}{2}
\\\noalign{\medskip}1
\\\noalign{\medskip}\frac{3}{2}
\\\noalign{\medskip}2
\\\noalign{\medskip}\end{array}&\begin{array}{c}0
\\\noalign{\medskip}0
\\\noalign{\medskip}1
\\\noalign{\medskip}1
\\\noalign{\medskip}\end{array}
&\begin{array}{c}0
\\\noalign{\medskip}1
\\\noalign{\medskip}2
\\\noalign{\medskip}4
\\\noalign{\medskip}\end{array}
&\begin{array}{c}0
\\\noalign{\medskip}\frac{1}{2}
\\\noalign{\medskip}1
\\\noalign{\medskip}2
\\\noalign{\medskip}\end{array}
\end{Bmatrix}
,\] where $\tilde{s}\in\mathbb{C}$ is an apparent singularity.
Next, we check that $\left(\frac{z}{z-1}\right)^{*}\Delta(P_1)\star_H I_{\frac{1}{2}}=LP_2$, where $L$ is the unique monic operator of degree one whose Riemann-scheme reads
\[\mathcal{R}(L)=\begin{Bmatrix}
0&1&(s_1-s_2)^2/4s_1s_2&\tilde{s}&\infty\\[0.1cm] \hline\\[-0.25cm]
\begin{array}{c} -3
\\\noalign{\medskip}\end{array}&\begin{array}{c} -1
\\\noalign{\medskip}\end{array}&\begin{array}{c} -3
\\\noalign{\medskip}\end{array}&\begin{array}{c} \frac{1}{2}
\\\noalign{\medskip}\end{array}&\begin{array}{c} \frac{13}{2}
\\\noalign{\medskip}\end{array}\end{Bmatrix}.\]
The tuple of Jordan forms associated to each monodromy tuple of
\[Q_4(s_1,s_2,c,\lambda)=\left(P_2\otimes \left(4sz+(s-1)^2\right)^{-1/2}\right)\star_H I_{1+\frac{\lambda}{2}}\star_H I_{1-\frac{\lambda}{2}}\] is similar to $M_4(\exp(\pi i\lambda))$.
Carrying out all computations explicitly yields

\begin{prop}
 For each $\lambda\in\mathbb{Q}\setminus\mathbb{Z}$ the tuple of Jordan forms associated to the monodromy tuple of 
\begin{align*}
Q_4(s_1,s_2,c,\lambda)&=4\vartheta^4(s_1 -s_2)^4-z(s_1-s_2)^2\left(4(s_1^2-10s_1s_2+s_2^2)(\vartheta^4+2\vartheta^3)\right)\\& -z(s_1-s_2)^2\vartheta^2\left(5(s_1^2+s_2^2)+2s_1s_2(4\lambda^2-45)+8c(s_1+s_2)\right)\\&-z(s_1-s_2)^2\vartheta\left(s_1^2+s_2^2+2s_1s_2(4\lambda^2-25)+8c(s_1+s_2)\right)\\&-z(s_1-s_2)^2\left(s_1s_2(3\lambda^2-11)+2c(s_1+s_2)+4c^2
\right)\\&-4z^2s_1s_2(2\vartheta+2+\lambda)(2\vartheta+2-\lambda)(2(s_1^2-4s_1s_2+s_2^2)(\vartheta^2+2\vartheta))\\&-4z^2s_1s_2(2\vartheta+2+\lambda)(2\vartheta+2-\lambda)
(3(s_1^2+s_2^2)+s_1s_2(\lambda^2-11)+2c(s_1+s_2))
\\&-4z^3s_1^2s_2^2(2\vartheta+2+\lambda)(2\vartheta+2-\lambda)(2\vartheta+4+\lambda)(2\vartheta+4-\lambda)
\end{align*}
is similar to $M_4(\exp(\pi i\lambda))$.
\end{prop}

Finally, if \[\J(T)=M_5(\alpha)=\left(J(4), J(1)^2\oplus -J(1)^2, J(1)^2\oplus -J(1)^2, [\alpha]\oplus -J(1)^2\right)\] we find that
\begin{align*}J'&=\MC_{-1}\left(M_5(\alpha)\right)\otimes K^4_1(-1)\\&=(J(3)\oplus -J(1)^2, J(2)^2\oplus J(1), J(2)^2\oplus J(1),[\alpha\oplus J(1)]\oplus J(1)^2).\end{align*}
By Proposition \ref{ConstMDSPconv}, $\MC_{-1}(T)$ lies in $\SO_5(\mathbb{C})$ and the exceptional isomorphism between $\SO_5(\mathbb{C})$ and $\Sp_4(\mathbb{C})$ allows us to write
$J'=\bigwedge^2(J'')$ for 
\[J''=\left(iJ(2)\oplus -iJ(2), J(2)\oplus J(1)^2, J(2)\oplus J(1)^2, \beta J(1)^2\oplus \beta^{-1}J(1)^2\right)\] with $\beta^2=\alpha$. Moreover, we find that
\begin{align*}
&\MC_{-i\beta}\left(\MC_{i\beta^{-1}}\left(J''\otimes K^4_1(i)\right)\otimes K_1^4(-i\beta)\right)\otimes K_4^1(-i\beta)\\&=(-J(2), -J(1)\oplus J(1), -J(1)\oplus J(1), J(2)). 
\end{align*}
To mimic the inverse construction on the level of differential operators, we start with $R\left(0,\frac{1}{2},\frac{1}{2}, 0, s_1, s_2, c\right)$. One checks directly that
\begin{align*}P_1&=\left(\left(R\left(0,\frac{1}{2}, \frac{1}{2}, 0,c,s\right)\otimes z^{-\frac{1}{2}}\right)\star_H I_{\frac{5}{4}-\left(\frac{1}{4}+\frac{a}{2}\right)}\star_H I_{\frac{1}{4}+\left(\frac{1}{4}+\frac{a}{2}\right)}\right)\otimes z^{\frac{3}{4}}\\& =s_{{1}}s_{{2}} \left( 4\,\vartheta-1 \right) ^{2} \left( 4\,\vartheta-
3 \right) ^{2}\\&-z \left( 4\,\vartheta-1+2\,a \right)  \left( 4\,
\vartheta+1-2\,a \right)  \left( 16\,{\vartheta}^{2}s_{{1}}+16\,{
\vartheta}^{2}s_{{2}}-s_{{1}}-s_{{2}}+16\,c \right)\\& +{z}^{2} \left( 2
\,a+3+4\,\vartheta \right)  \left( -2\,a+5+4\,\vartheta \right) 
 \left( 4\,\vartheta-1+2\,a \right)  \left( 4\,\vartheta+1-2\,a
 \right) 
\end{align*}
is self-dual and hence that the tuple of Jordan forms associated to the monodromy tuple of $\bigwedge^2(P_1)$ coincides with $J'$. The desired operator then reads $Q_5(s_1,s_2,c,\lambda)=\bigwedge^2(P_1)\star_H I_{\frac{3}{2}}$. Carrying out the computations explicitly yields

\begin{prop}
For each $a\in\mathbb{Q}\setminus\mathbb{Z}$ the tuple of Jordan forms associated to each monodromy tuple of 

\begin{align*}
Q_5(s_1,s_2,c,\lambda)&=16s_1^2s_2^2\vartheta^4-4s_1s_2z(8(s_1+s_2)(\vartheta^4+2\vartheta^3))\\&-4s_1s_2z\vartheta^2\left(2(s_1+s_2)(a-a^2+9)+8c\right)\\&-4s_1s_2z\vartheta\left(2(s_1+s_2)(a-a^2+5)+8c
\right)\\&-4s_1s_2z\left((s_1+s_2)(a-a^2+2)+4c(a^2-a+1)
\right)\\&+z^2\left(16(s_1^2+4s_1s_2+s_2^2)(\vartheta^4+4\vartheta^3)\right)\\&+z^2\vartheta^2\left(4(s_1^2+s_2^2)(2a-2a^2+23)+32s_1s_2(a-a^2+15)+32c(s_1+s_2)\right)\\&+z^2\vartheta\left(8(s_1^2+s_2^2)(2a-2a^2+7)+64s_1s_2(a-a^2+7)+64c(s_1+s_2)\right)\\&+z^2\left((s_1^2+s_2^2)(a^4-2a^3-7a^2+8a+12)+2s_1s_2(84+20a-21a^2-a^4+2a^3)\right)\\&+z^2\left(8c(s_1+s_2)(a^2-a+4)+16c^2\right)\\&-2z^3(2\vartheta+3)^2\left(4(s_1+s_2)(\vartheta^2+3\vartheta)\right)\\&-2z^3(2\vartheta+3)^2\left((s_1+s_2)(3a-3a^2+11)+4c
\right)\\&+z^4(2\vartheta+3)(2\vartheta+5)(2\vartheta+3+2a)(2\vartheta+5-2a)
\end{align*}
is similar to $M_5(\exp(\pi ia))$.
\end{prop}

As a result of the constructions stated in this section, we get

\begin{thm}\label{resumee}
For each tuple of Jordan forms $M_i(\alpha)$ stated in Proposition \ref{Classtup}, we find a three-parameter family of irreducible fuchsian differential operators whose tuples of Jordan forms associated to their monodromy tuples are similar to $M_i(\alpha)$. Moreover, each of the operators $L$ we found is self-dual in the sense that there is an $0\neq \alpha\in\mathbb{Q}(z)$ such that $L\alpha=\alpha L^{\vee}$ holds, where $L^{\vee}$ denotes the adjoint of $L$, and has only zero as exponent at $z=0$.
\end{thm}

\section{Identification of CY-type operators}
\subsection{How to find CY-type operators?}

We briefly recall the definition of a differential operator of CY-type:

\begin{defin}\label{CYdefin}
An irreducible differential operator $L\in\mathbb{Q}[z,\vartheta]$ is called \textbf{of CY-type} if it fulfills each of the following CY-properties:
\begin{enumerate}
 \item[(P)] $L$ is self-dual in the sense that there is an $0\neq\alpha\in\mathbb{Q}(z)$ such that $L\alpha=\alpha L^{\vee}$ holds.
 \item[(M)] All exponents of $L$ at $z=0$ are integers and equal.
 \item[(N)] $L$ has at $z=0$ an \textit{N-integral} local solution $y_0\in\mathbb{Q}\llbracket z\rrbracket$, i.e. there is an $N\in\mathbb{N}$ such that $N^mA_m\in\mathbb{Z}$ for all $m\in\mathbb{N}$.
 \item[(Q)] The \textit{special coordinate} of $L$ at $z=0$, i.e. \[q=e^{\frac{y_1}{y_0}}\in z\mathbb{Q}\llbracket z\rrbracket,\] where $y_1=\ln(z)y_0+r$, $r\in z\mathbb{Q}\llbracket z\rrbracket$ is a local solution of $L$, is N-integral.
 \item[(S)] Each of the the formal power series $\alpha_1,\dots,\alpha_n$ appearing in the local expression
\[Ly_0:=\vartheta\alpha_n\vartheta\alpha_{n-1}\cdots\alpha_1\vartheta\in\mathbb{Q}\llbracket z\rrbracket[\vartheta]\] are N-integral.
 \end{enumerate}
\end{defin}

It is clear that each M-tuple attached to a differential operator of CY-type has those properties.
For each family of operators constructed in the previous section, we try to detect those which are of CY-type. Note that each operator of families $P_{a,1}$ and $P_{a,2}$ stated in Proposition \ref{Dreieins} satisfies CY-properties (P), (M) and (N). Hence it remains to check experimentally whether (Q) and (S) hold. The results are stated in Section \ref{Opthree}.  
The other constructions we did rely on the choice of differential operators of order two. 
As stated in Theorem \ref{resumee}, each of the operators $Q_1-Q_5$ satisfies properties (M) and (P). To discuss the remaining CY-properties, we first observe that all of them depend on the choices of $s_1, s_2$ and $c$ and that none of them seems to hold for an arbitrary choice. 
A closer look at the constructions on the level of differential operators we used reveals, that all of them preserve N-integrality of solutions which are locally holomorphic near $z=0$.
Therefore, a first step towards a strategy to find CY-type operators is to guarantee, that the Heun operator $R(h)$ we start with admits an N-integral solution at $z=0$. Moreover - for reasons which are not understood yet - each constructed operator of order four seems to satisfy property (Q) if the Heun operator we started with does. Therefore, we chose the parameters $s_1, s_2, c$ in such a way that the Heun operator underlying the construction is of CY-type. All examples of this type we know can be obtained by algebraic pullbacks of hypergeometric differential operators. The majority of them are Picard-Fuchs operators for families of relatively minimal elliptic curves with section over $\mathbb{Q}$ and are obtained from \cite{Her}.

\subsection{Underlying CY-type operators of degree two}

In the upcoming tables, we state those Heun operators which are of CY-type and are suitable for the constructions of $Q_1-Q_5$. We only state the parameter $c$ explicitly, as the singularities $s_1$ and $s_2$ can be regained directly as the roots of the polynomial coefficient of $\vartheta^2$. 

For $\sign(R(h))=(0,0,0,0)$, we find the following six operators.

\setlength{\extrarowheight}{0.5mm}
\begin{longtable}{c c c}
\textbf{Number}&\textbf{Operator}&\textbf{c}\\\hline
$1$&${\vartheta}^{2}-z \left(7\,{\vartheta}^{2}+7\,\vartheta+2 \right) -8
\,{z}^{2} \left( \vartheta+1 \right) ^{2}$
&$1/4$ \\[0.5mm]
$2$&${\vartheta}^{2}-z \left( 11\,{\vartheta}^{2}+11\,\vartheta+3 \right) 
-{z}^{2} \left( \vartheta+1 \right) ^{2}
$&$-3$ \\[0.5mm]
$3$&${\vartheta}^{2}-z \left( 10\,{\vartheta}^{2}+10\,\vartheta+3 \right) 
+9\,{z}^{2} \left( \vartheta+1 \right) ^{2}
$&$1/3$\\[0.5mm]
$4$&${\vartheta}^{2}-4z \left(3\,{\vartheta}^{2}+3\,\vartheta+1 \right) 
+32\,{z}^{2} \left( \vartheta+1 \right) ^{2}
$&$1/8$\\[0.5mm]
$5$&${\vartheta}^{2}-3z \left(3\,{\vartheta}^{2}+3\,\vartheta+1 \right) +
27\,{z}^{2} \left( \vartheta+1 \right) ^{2}
$&$1/9$\\[0.5mm]
$6$&${\vartheta}^{2}-z \left(17\,{\vartheta}^{2}+17\,\vartheta+6 \right) 
+72\,{z}^{2} \left( \vartheta+1 \right) ^{2}
$&$1/12$\\[0.5mm]
\end{longtable}

\newpage

For $\sign(L)=(0,\lambda,\lambda,0)$ we find

\setlength{\extrarowheight}{0.5mm}
\begin{longtable}{c c c c}
\textbf{Nr.}& $\lambda$&\textbf{Operator}&\textbf{c}\\\hline 
$1'$& $1/2$&$4\,{\vartheta}^{2}-2z \left(28\,{\vartheta}^{2}+14\,\vartheta+3
 \right) +81\,{z}^{2} \left( 2\,\vartheta+1 \right) ^{2}$
&$1/54$\\[0.5mm]

$2'$&$1/2$&$4\,{\vartheta}^{2}-2z \left(44\,{\vartheta}^{2}+22\,\vartheta
 +5\right) +125\,{z}^{2} \left( 2\,\vartheta+1 \right) ^{2}
$&$1/50$\\[0.5mm]

$3'$&$1/2$&${\vartheta}^{2}-2z \left(10\,{\vartheta}^{2}+5\,\vartheta+1 \right) 
+16\,{z}^{2} \left( 2\,\vartheta+1 \right) ^{2}
$&$1/32$\\[0.5mm]

$4'$&$1/2$&$\vartheta^2-2z(12\vartheta^2+6\vartheta+1)+4z^2(2\vartheta+1)^2$&$1/8$\\[0.5mm]

$5'$&$1/2$&$4\,{\vartheta}^{2}+6z \left(12\,{\vartheta}^{2}+6\,\vartheta+1
 \right) -27\,{z}^{2} \left( 2\,\vartheta+1 \right) ^{2}
$&$1/18$\\[0.5mm]

$6'$&$1/2$&$4\,{\vartheta}^{2}-2z \left(68\,{\vartheta}^{2}+34\,\vartheta+5
 \right) +{z}^{2} \left( 2\,\vartheta+1 \right) ^{2}
$&$5/2$\\[0.5mm]

$7$&$1/3$&$9\,{\vartheta}^{2}-3z \left(39\,{\vartheta}^{2}+26\,\vartheta+7
 \right) +49\,{z}^{2} \left( 3\,\vartheta+2 \right) ^{2}
$&$1/21$\\[0.5mm]

$8$&$1/3$&$9\,{\vartheta}^{2}+12z \left(15\,{\vartheta}^{2}+10\,\vartheta
+2 \right) -8\,{z}^{2} \left( 3\,\vartheta+2 \right) ^{2}
$&$1/3$\\[0.5mm]

$9$&$1/4$&$16\vartheta^2-4z(24\vartheta^2+18\vartheta+5)+25\,{z}^{2} \left( 4\,\vartheta+3 \right) ^{2}$&$1/20$\\[0.5mm]

$10$&$1/4$&$16\vartheta^2-4z(56\vartheta^2+42\vartheta+9)+\,{z}^{2} \left( 4\,\vartheta+3 \right) ^{2}$&$9/4$\\[0.5mm]
\end{longtable}

In the case $\sign(L)=\left(0,\frac{1}{2},\frac{1}{2},\lambda\right)$, we find

\setlength{\extrarowheight}{0.5mm}
\begin{longtable}{c c c c}
\textbf{Nr.}& $\lambda$&\textbf{Operator}&\textbf{c}\\\hline 
$1''$&$1/2$&${\vartheta}^{2}+2z \left(14\,{\vartheta}^{2}+7\,\vartheta +1\right) -
8\,{z}^{2} \left( 4\,\vartheta+1 \right)  \left( 4\,\vartheta+3
 \right)$&$1/64$\\[0.5mm]

$2''$&$1/2$&${\vartheta}^{2}+z \left(44\,{\vartheta}^{2}+22\,\vartheta+3 \right) -
{z}^{2} \left( 4\,\vartheta+1 \right)  \left( 4\,\vartheta+3 \right) 
$&$3/16$\\[0.5mm] 

$3''$&$1/2$&${\vartheta}^{2}-z \left( 40\,{\vartheta}^{2}+20\,\vartheta+3 \right) 
+9\,{z}^{2} \left( 4\,\vartheta+1 \right)  \left( 4\,\vartheta+3
 \right)$&$1/48$\\[0.5mm]

$4''$&$1/2$&${\vartheta}^{2}-4z \left(12\,{\vartheta}^{2}+6\,\vartheta +1\right) 
+32\,{z}^{2} \left( 4\,\vartheta+1 \right)  \left( 4\,\vartheta+3
 \right) 
$&$1/128$\\[0.5mm]

$5''$&$1/2$&$\,{\vartheta}^{2}+3z \left( 12\,{\vartheta}^{2}+6\,\vartheta+1
 \right) +27\,{z}^{2} \left( 4\,\vartheta+1 \right)  \left( 4\,
\vartheta+3 \right)$&$-1/144$\\[0.5mm]

$6''$&$1/2$&${\vartheta}^{2}-2z \left(34\,{\vartheta}^{2}+17\,\vartheta+3 \right) 
+72\,{z}^{2} \left( 4\,\vartheta+1 \right)  \left( 4\,\vartheta+3
 \right)$&$1/192$\\[0.5mm]

$7'$&$1/3$&${\vartheta}^{2}+z \left(26\,{\vartheta}^{2}+13\,\vartheta+2 \right) -
3\,{z}^{2} \left( 3\,\vartheta+1 \right)  \left( 3\,\vartheta+2
 \right)$&$2/27$\\[0.5mm]

$8'$&$1/3$&${\vartheta}^{2}-4z \left(10\,{\vartheta}^{2}+5\,\vartheta+1 \right) 
+48\,{z}^{2} \left( 3\,\vartheta+1 \right)  \left( 3\,\vartheta+2
 \right) 
$&$1/108$\\[0.5mm]

$9'$&$1/4$&${\vartheta}^{2}-z \left(12\,{\vartheta}^{2}+6\,\vartheta+1 \right) -
{z}^{2} \left( 8\,\vartheta+3 \right)  \left( 8\,\vartheta+5 \right) 
$&$-1/64$\\[0.5mm]

$10'$&$1/4$&${\vartheta}^{2}-z \left(28\,{\vartheta}^{2}+3\,\vartheta+3 \right) +
12{z}^{2} \left( 8\,\vartheta+3 \right)  \left( 8\,\vartheta+5 \right) 
$&$1/256$\\[0.5mm]
\end{longtable}

The operators with signature $(0,\frac{1}{2},\frac{1}{2},\lambda)$ can be constructed as follows:
Consider a Heun operator $L_1$ with $\sign(L_1)=(0,\lambda,\lambda,0)$ such that $s_1$ and $s_2$ are roots of $p(X)=aX^2+bX+1\in\mathbb{Z}[X]$. 
Then \[L_2=\left(\left(-z/p(z)\right)^{\vee}\right)^{*}L_1\in\mathbb{Q}[z,\vartheta]\] has four singularities $\{0,\tilde{s}_1,\tilde{s}_2,\infty\}$ and $\sign(L_2)=(0,\frac{1}{2},\frac{1}{2},\lambda)$ holds. Moreover, $L_2$ is of CY-type, if $L_1$ is. We referred to this construction by decorating the number of the operator with a $(\cdot)'$.

Up to a tensor product with an operator of degree one, the following operators are Picard-Fuchs operators for the following families $Y^2=4X^3-g_2(z)X-g_3(z)$ of elliptic curves:

\setlength{\extrarowheight}{0.5mm}
\begin{longtable}{c c c}
\textbf{Number}&\multicolumn{2}{c}{\textbf{Families}}\\[0.5mm]\hline\hline
\multirow{2}{*}{$1$}& {$g_2$}&{$3\left( 4\,z+1 \right) \left( 64\,{z}^{3}+48\,{z}^{2}-12\,z+1
 \right)$}\\[0.5mm]
&$g_3$&$\left( 8\,{z}^{2}+4\,z-1 \right)  \left( 512\,{z}^{4}+512\,{z}^{3}
+8\,z-1 \right)$\\[0.5mm]\hline
\multirow{2}{*}{$2$}&{$g_2$}& {$3(z^4-12z^3+14z^2+12z+1)$}\\[0.5mm]
&$g_3$&${z}^{6}-18\,{z}^{5}+75\,{z}^{4}+75\,{z}^{2}+18\,z+1$\\[0.5mm]\hline
\multirow{2}{*}{$3$}&{$g_2$}&{$3\, \left( 3\,z+1 \right)  \left( 243\,{z}^{3}+243\,{z}^
{2}+9\,z+1 \right)$}\\[0.5mm]
&$g_3$& $\left( 27\,{z}^{2}+18\,z-1 \right)  \left( 729\,{z}^
{4}+972\,{z}^{3}+270\,{z}^{2}+36\,z+1 \right)$\\[0.5mm]\hline 
\multirow{2}{*}{$4$}&{$g_2$}&{$3(65536\,{z}^{4}-32768\,{z}^{3}+5120\,{z}^{2}-256\,z+1)$} \\[0.5mm]
&$g_3$&$\left( 128\,{z}^{2}-32\,z+1 \right)  \left( 131072\,{z}^{4}-65536\,{z}^{3}+10240\,{z}^{2}-512\,z-1 \right)$\\[0.5mm]\hline
\multirow{2}{*}{$5$}&{$g_2$}&$3\, \left( 9\,z-1 \right)  \left( 6561\,{z}^{3}-2187\,{z}^{2}+243\,z-1
 \right)$\\[0.5mm]
&$g_3$&$14348907\,{z}^{6}-9565938\,{z}^{5}+2657205\,{z}^{4}-367416\,{z}^{3}+
24057\,{z}^{2}-486\,z-1$\\[0.5mm]\hline
\multirow{2}{*}{$6$}&{$g_2$}&$108\, \left( 12\,z -1\right)  \left( 15552\,{z}^{3}-3888\,{z}^{2}+252\,z-1 \right)$\\[0.5mm]
&$g_3$&$216\, \left( 216\,{z}^{2}-36\,z+1 \right)  \left( 373248\,{z}^{4}-
124416\,{z}^{3}+13824\,{z}^{2}-504\,z-1 \right)$\\[0.5mm]\hline
\multirow{2}{*}{$7$}&{$g_2$}&{$147\, \left( 196\,{z}^{2}-26\,z+1 \right)  \left( 9604\,{z}^{2}-490\,z
+1 \right)$}\\[0.5mm]
&$g_3$&$343\, \left( 196\,{z}^{2}-26\,z+1 \right)  \left( 13176688\,{z}^{4}-
1882384\,{z}^{3}+86436\,{z}^{2}-980\,z-1 \right)$\\[0.5mm]\hline
\multirow{2}{*}{$8$}&{$g_2$}&{$3(64z^4-192z^3+64z^2+24z+1)$}\\[0.5mm]
&$g_3$&$512\,{z}^{6}-2304\,{z}^{5}+2496\,{z}^{4}+312\,{z}^{2}+36\,z+1$\\[0.5mm]\hline
\multirow{2}{*}{$2'$}&{$g_2$}&{$3\, \left( 500\,{z}^{2}+44\,z+1 \right)  \left( 20\,{z}^{2}+20\,z+1
 \right)$}\\[0.5mm]
&$g_3$&$- \left( 500\,{z}^{2}+44\,z+1 \right) ^{2} \left( 4\,{z}^{2}-8\,z-1
 \right)$\\[0.5mm]\hline
\multirow{2}{*}{$5'$}&{$g_2$}&{$ \left( 81\,{z}^{2}-54\,z-3 \right)  \left( 9\,z+1 \right)  \left( 3\,
z-1 \right) 
$}\\[0.5mm]
&$g_3$&$- \left( 729\,{z}^{4}-972\,{z}^{3}+270\,{z}^{2}+36\,z+1 \right) 
 \left( 27\,{z}^{2}+1 \right) 
$\\[0.5mm]\hline
\multirow{2}{*}{$7'$}&{$g_2$}&{$-3\, \left( 432\,{z}^{2}-40\,z+1 \right)  \left( 16\,z-1 \right) $}\\[0.5mm]
&$g_3$&$\left( 432\,{z}^{2}-40\,z+1 \right) ^{2} \left( 4\,z-1 \right)$\\[0.5mm]\hline
\end{longtable}

\subsection{Resulting operators of order four with three singularities}\label{Opthree}
We state for each of the families the operators of CY-type inside we found with the methods described before. For each of these operators, we have normalized the coordinate via $z\mapsto \lambda z$ in such a way that its q-coordinate lies in $z\mathbb{Z}\llbracket z\rrbracket$ and is minimal in the sense that there is no $N\in\mathbb{N}$ such that $N^m$ divides the $m$-th coefficient of the q-coordinate for all $m\in\mathbb{N}$. To identify the operators with previously known examples as collected in \cite[Appendix A]{AESZ}, the column \texttt{source} provides their first three genus zero instanton numbers as defined there. If the instanton numbers of two CY-type operators coincide, they can be locally transformed into each other, see e.g. \cite[Proposition 3.4.9]{Diss}. 

In case that all previously known operators with the same instanton numbers have more singularities than the one stated here, the numbers are written in \textit{italics}. In case that we know an operator having the same instanton as the one we found but a smaller number singularities, we write them {\scriptsize {smaller}}. If the numbers are written \textbf{boldly}, none of the previously known operators has the same instanton numbers.

\textbf{Family:}\quad $P_{a,1}\otimes (1-z)^{-a}$, where
\begin{align*}
P_{a,1}&:=\left(\left(S_a\otimes z^{-\frac{1}{4}-a}\right)\star I_{\frac{5}{4}+a}\otimes z^{\frac{1}{4}+a} (1-z)^{a-\frac{1}{4}}\right)\star I_{\frac{3}{4}-a}\\& =16\,{\vartheta}^{4}-z\left(32\vartheta^4+64(1-a)\vartheta^3+\left(63-96a-112a^2\right)\vartheta^2\right)\\&\quad-z\left(\left(31-64a-112a^2\right)\vartheta+6-16a-32a^2\right)\\&\quad+{z}^{2} \left( 4\,\vartheta+5-4\,a
 \right)  \left( \vartheta+1-4\,a \right)  \left( \vartheta+1+2\,a
 \right)  \left( 4\,\vartheta+3-4\,a \right)
\end{align*}

\textbf{Riemann scheme:}

\[\mathcal{R}(P_{a,1}\otimes (1-z)^{-a})=\begin{Bmatrix}
0&1&\infty\\[0.1cm] \hline\\[-0.25cm]\begin{array}{c} 0
\\\noalign{\medskip}0
\\\noalign{\medskip}0
\\\noalign{\medskip}0
\\\noalign{\medskip} \end{array}&\begin{array}{c} -a
\\\noalign{\medskip}a
\\\noalign{\medskip}1-a
\\\noalign{\medskip}1+a
\\\noalign{\medskip} \end{array}&\begin{array}{c} 1-3a
\\\noalign{\medskip}\frac{3}{4}
\\\noalign{\medskip}\frac{5}{4}
\\\noalign{\medskip}1+3a
\\\noalign{\medskip}\end{array}
\end{Bmatrix}\]
\newpage
\textbf{CY-operators:} We seem to get operators of CY-type in the following cases:
\begin{longtable}{c c c c c}
\textbf{a}&$0$&$\frac{1}{6}$&$\frac{1}{8}$&$\frac{1}{12}$\\
\textbf{source}& $-4, -11, -44$& $60,-7635,307860$& \textbf{-12, -186, -1668}& $-6, -33, -170$   
\end{longtable}

Moreover, replacing $a$ by $-a$ yields operators with the same instanton numbers.

\textbf{Family:}\quad$P_{2,a}\otimes (1-z)^{-3a}$, where
\begin{align*}P_{a,2}&:=\left(\left(R_a\otimes z^{-\frac{1}{4}-3a}\right)\star I_{\frac{5}{4}+3a}\otimes z^{3a+\frac{1}{4}}(1-z)^{3a-\frac{1}{4}}\right)\star I_{\frac{3}{4}-3a}\\&\quad16\,{\vartheta}^{4}+z \left( -32\,{\vartheta}^{4}-64\,{\vartheta}^{3}+
192\,{\vartheta}^{3}a-63\,{\vartheta}^{2}-272\,{\vartheta}^{2}{a}^{2}+
288\,{\vartheta}^{2}a\right)\\&\quad+z\left(\,192\vartheta\,a-272\,\vartheta\,{a}^{2}-31\,
\vartheta-6+48\,a-96\,{a}^{2} \right)\\&\quad +{z}^{2} \left( 4\,\vartheta+3-
12\,a \right)  \left( \vartheta+1-2\,a \right)  \left( \vartheta+1-4\,
a \right)  \left( 4\,\vartheta+5-12\,a \right) 
\end{align*}

\textbf{Riemann scheme:}
\[\mathcal{R}(P_{a,2}\otimes (1-z)^{-3a})=\begin{Bmatrix}
0&1&\infty\\[0.1cm] \hline\\[-0.25cm]\begin{array}{c} 0
\\\noalign{\medskip}0
\\\noalign{\medskip}0
\\\noalign{\medskip}0
\\\noalign{\medskip} \end{array}&\begin{array}{c} -3a
\\\noalign{\medskip}3a
\\\noalign{\medskip}1-3a
\\\noalign{\medskip}1+3a
\\\noalign{\medskip} \end{array}&\begin{array}{c} 1-a
\\\noalign{\medskip}\frac{3}{4}
\\\noalign{\medskip}\frac{5}{4}
\\\noalign{\medskip}1+a
\\\noalign{\medskip}\end{array}
\end{Bmatrix}\]

\textbf{CY-operators:} We seem to get operators of CY-type in the following cases:
\begin{longtable}{c c c c}
\textbf{a}&$0$&$\frac{1}{6}$&$\frac{1}{8}$\\
\textbf{source}& $-4, -11, -44$& {\scriptsize {60,-7635,307860}}& \textbf{20, 290, 28820/3}   
\end{longtable}
Again, replacing $a$ by $-a$ yields operators with the same instanton numbers.

Irreducible right factors $P_3(R)$ of $\Delta(R)\star_H I_{\frac{1}{2}}$, where the tuple of Jordan forms associated to the monodromy tuple of $R$ is similar to $J=(J(2),[\gamma], [\gamma], [\delta])$, the non-apparent singularities of $R$ are $\{0,-s,s,\infty\}$ and $R\not\in\mathbb{C}\left[z^2,\vartheta\right]$.
We only found CY-type operators of order four in the following two cases:\\

\textbf{First case:}
$R={\vartheta}^{2}+z \left( 4-32\,{\vartheta}^{2}+32\,\vartheta \right) +
{z}^{2} \left( 640+1024\,{\vartheta}^{2}+2048\,\vartheta \right) -8192
\,{z}^{3} \left( 2\,\vartheta+1 \right) ^{2}$

\begin{align*}&P_3(R)={\vartheta}^{4}+z \left( 656-4096\,{\vartheta}^{4}+24576\,{\vartheta}^
{3}+17408\,{\vartheta}^{2}+5120\,\vartheta \right)\\& +{z}^{2} \left( 
8716288-33554432\,{\vartheta}^{4}-134217728\,{\vartheta}^{3}+46137344
\,{\vartheta}^{2}+41943040\,\vartheta \right)\\& -{z}^{3} \left(
57176752128-137438953472\,{\vartheta}^{4}+274877906944\,{\vartheta}^{3
}+807453851648\,{\vartheta}^{2}+360777252864\,\vartheta \right)\\& +{z}^{
4} \left( 281474976710656\,{\vartheta}^{4}+
2251799813685248\,{\vartheta}^{3}+1759218604441600\,{\vartheta}^{2}\right)\\&+z^4\left(
562949953421312\,\vartheta +60473139527680\right)\\& -72057594037927936\,{z}^{5} \left( 
2\,\vartheta+1 \right) ^{4}
\end{align*}

The first three genus zero instanton numbers read $-3488, -1406056, -1142687008$.
\newpage
\textbf{Second case:}
$R=\vartheta^{2}+z\left(36-288({\vartheta}^{2}+\vartheta)\right) -{z}^{2} \left(11904+27648\,{\vartheta}^{2}+55296\vartheta \right)+884736\,{z}^{3} \left( 3\,\vartheta+1 \right)\left( 3\,\vartheta+2 \right)$

\begin{align*}
&P_3(R)={\vartheta}^{4}+z \left( -10608-884736\,{\vartheta}^{4}+221184\,{
\vartheta}^{3}+49152\,{\vartheta}^{2}-61440\,\vartheta \right)\\& +{z}^{2
} \left( -350355456+269072990208\,{\vartheta}^{4}-24461180928\,{
\vartheta}^{3}+124004597760\,{\vartheta}^{2}+15627976704\,\vartheta
 \right)\\& -{z}^{3} \left( 32462531054272512\,{
\vartheta}^{4}+24346898290704384\,{\vartheta}^{3}+20063647665487872\,{
\vartheta}^{2}\right)\\&-z^3\left(5748573207527424\,\vartheta+632272604626944 \right)\\& +
37396835774521933824\,{z}^{4} \left( 3\,\vartheta+1 \right)  \left( 2
\,\vartheta+1 \right) ^{2} \left( 3\,\vartheta+2 \right) 
\end{align*}

The first three genus zero instanton numbers read $-188832, -3134817768, -101990911789344$.

\subsection{Resulting operators of order four with four singularities}\label{Opfour}

We state for each of the families the operators of CY-type inside we found with the methods described before. For each of these operators, we have normalized the coordinate via $z\mapsto \lambda z$ in such a way that its q-coordinate lies in $z\mathbb{Z}\llbracket z\rrbracket$ and is minimal in the sense that there is no $N\in\mathbb{N}$ such that $N^m$ divides the $m$-th coefficient of the q-coordinate for all $m\in\mathbb{N}$. To identify the operators with previously known examples as collected in \cite[Appendix A]{AESZ}, the column \texttt{source} provides their first three genus zero instanton numbers as defined there. If the instanton numbers of two CY-type operators coincide, they can be locally transformed into each other, see e.g. \cite[Proposition 3.4.9]{Diss}. In case that all previously known operators with the same instanton numbers have more singularities than the one stated here, the numbers are written in \textit{italics}. If the numbers are written \textbf{boldly}, none of the previously known operators has the same instanton numbers.

\textbf{Family:}
\begin{align*}
&Q_1(c,s,a)={\vartheta}^{4}s-z \left( \vartheta+a \right)  \left( \vartheta+1-a
 \right)  \left( {\vartheta}^{2}(1+s)+\vartheta(1+s)+c \right)\\&\quad +{z}^{2} \left( \vartheta+2-a \right)  \left( \vartheta+
1-a \right)  \left( \vartheta+1+a \right)  \left( \vartheta+a \right) 
\end{align*}  

\textbf{Riemann scheme:} \[\begin{Bmatrix}
0&s_1&s_2&\infty\\[0.1cm] \hline\\[-0.25cm]\begin{array}{c} 0
\\\noalign{\medskip}0
\\\noalign{\medskip}0
\\\noalign{\medskip}0
\\\noalign{\medskip} \end{array}&\begin{array}{c} 0
\\\noalign{\medskip}1
\\\noalign{\medskip}1
\\\noalign{\medskip}2
\\\noalign{\medskip} \end{array}&\begin{array}{c} 0
\\\noalign{\medskip}1
\\\noalign{\medskip}1
\\\noalign{\medskip}2
\\\noalign{\medskip}\end{array}&\begin{array}{c} a
\\\noalign{\medskip}1-a
\\\noalign{\medskip}1+a
\\\noalign{\medskip}2-a
\\\noalign{\medskip}\end{array}
\end{Bmatrix}.\]

\textbf{CY-operators:} We seem to get operators of CY-type if $a\in\left\{\frac{1}{2},\frac{1}{3},\frac{1}{4},\frac{1}{6}\right\}$. In particular, we find
\setlength{\extrarowheight}{0.5mm}
\begin{longtable}{c c c}
\textbf{Nr.}&\textbf{Operator}&\textbf{Source}\\\hline\hline
\multirow{4}{*}{$1$}&{\multirow{4}{*}{\begin{minipage}{1cm}\centering\begin{align*}{\vartheta}^{4}-z \left( 7\,{\vartheta}^{2}+7\,\vartheta+2 \right) 
 \left( \vartheta+a \right)  \left( \vartheta+1-a \right)\\ -8\,{z}^{2}
 \left( \vartheta+2-a \right)  \left( \vartheta+1-a \right)  \left( 
\vartheta+1+a \right)  \left( \vartheta+a \right)\end{align*}\end{minipage}}}&$a=\frac{1}{2}:\ 12,163,3204$
\\[0.5mm]
&&$a=\frac{1}{3}:\ 21, 480, 15894$\\[0.5mm]
&&$a=\frac{1}{4}:\ 52, 2814, 220220$\\[0.5mm]
&&$a=\frac{1}{6}:\ 372, 136182, 71562236$\\[0.5mm]\hline

\multirow{4}{*}{$2$}&{\multirow{4}{*}{\begin{minipage}{1cm}\centering\begin{align*}{\vartheta}^{4}-z \left( 11\,{\vartheta}^{2}+11\,\vartheta+3 \right) 
 \left( \vartheta+a \right)  \left( \vartheta+1-a \right)\\ -{z}^{2}
 \left( \vartheta+2-a \right)  \left( \vartheta+1-a \right)  \left( 
\vartheta+1+a \right)  \left( \vartheta+a \right) 
\end{align*}\end{minipage}}}&{$a=\frac{1}{2}:\ 20, 277, 8220$}\\[0.5mm] 
&&$a=\frac{1}{3}:\ 36, 837, 41421$\\[0.5mm]
&&$a=\frac{1}{4}:\ 92, 5052, 585396$\\[0.5mm]
&&$a=\frac{1}{6}:\ 684, 253314, 195638820$\\[0.5mm]\hline\newpage

\multirow{4}{*}{$3$}&{\multirow{4}{*}{\begin{minipage}{1cm}\centering\begin{align*}{\vartheta}^{4}-z \left( 10\,{\vartheta}^{2}+10\,\vartheta+3 \right) 
 \left( \vartheta+a \right)  \left( \vartheta+1-a \right)\\ +9\,{z}^{2}
 \left( \vartheta+2-a \right)  \left( \vartheta+1-a \right)  \left( 
\vartheta+1+a \right)  \left( \vartheta+a \right) 
\end{align*}\end{minipage}}}&{$a=\frac{1}{2}:\ 16, 142, 11056/3$}\\[0.5mm]
&&$a=\frac{1}{3}:\ 27, 432, 18089$\\[0.5mm]
&&$a=\frac{1}{4}:\ 64, 2616, 246848$\\[0.5mm]
&&$a=\frac{1}{6}:\ 432, 130842, 78259376$\\[0.5mm]\hline

\multirow{4}{*}{$4$}&{\multirow{4}{*}{\begin{minipage}{1cm}\centering\begin{align*}{\vartheta}^{4}-4\,z \left( 3\,{\vartheta}^{2}+3\,\vartheta+1 \right) 
 \left( \vartheta+a \right)  \left( \vartheta+1-a \right)\\ +32\,{z}^{2}
 \left( \vartheta+2-a \right)  \left( \vartheta+1-a \right)  \left( 
\vartheta+1+a \right)  \left( \vartheta+a \right) 
\end{align*}\end{minipage}}}&{$a=\frac{1}{2}:\ 16, 42, 1232$}\\[0.5mm]
&&$a=\frac{1}{3}:\ 24, 291/2, 5832$\\[0.5mm]
&&$a=\frac{1}{4}:\ 48, 998, 73328$\\[0.5mm]
&&$a=\frac{1}{6}:\ 240, 57102, 19105840$\\[0.5mm]\hline

\multirow{4}{*}{$5$}&{\multirow{4}{*}{\begin{minipage}{1cm}\centering\begin{align*}{\vartheta}^{4}-3\,z \left( 3\,{\vartheta}^{2}+3\,\vartheta+1 \right) 
 \left( \vartheta+a \right)  \left( \vartheta+1-a \right)\\ +27\,{z}^{2}
 \left( \vartheta+2-a \right)  \left( \vartheta+1-a \right)  \left( 
\vartheta+1+a \right)  \left( \vartheta+a \right) 
\end{align*}\end{minipage}}}&$a=\frac{1}{2}:\ 12, -42, -3284/3$\\[0.5mm]
&&$a=\frac{1}{3}:\ 18,-207/2,-5177$\\[0.5mm]
&&$a=\frac{1}{4}:\ 36, -477, -206716/3$\\[0.5mm]
&&$a=\frac{1}{6}:\ 180, -15615, -21847076$\\[0.5mm]\hline

\multirow{4}{*}{$6$}&{\multirow{4}{*}{\begin{minipage}{1cm}\centering\begin{align*}{\vartheta}^{4}-z \left( 17\,{\vartheta}^{2}+17\,\vartheta+6 \right) 
 \left( \vartheta+a \right)  \left( \vartheta+1-a \right)\\ +72\,{z}^{2}
 \left( \vartheta+2-a \right)  \left( \vartheta+1-a \right)  \left( 
\vartheta+1+a \right)  \left( \vartheta+a \right) 
\end{align*}\end{minipage}}}&{$a=\frac{1}{2}:\ 20, 2, 1684/3$}\\[0.5mm]
&&$a=\frac{1}{3}:\ 27, 189/4, 2618$\\[0.5mm]
&&$a=\frac{1}{4}:\ 44, 607, 22500$\\[0.5mm]
&&$a=\frac{1}{6}:\ 108, 54135, -4945756$\\[0.5mm]\hline
\end{longtable} 

\textbf{Family:}
\begin{align*}
&Q_2(c,s,\lambda)=4\,{\vartheta}^{4}s-z \left( 2\,\vartheta+1 \right) ^{2} \left( {
\vartheta}^{2}+{\vartheta}^{2}s+\vartheta+\vartheta\,s+4\,c \right)\\&\quad +{
z}^{2} \left( 2\,\vartheta+3 \right)  \left( 2\,\vartheta+1 \right) 
 \left( \vartheta+1+\lambda \right)  \left( \vartheta+1-\lambda \right), 
\end{align*}

\textbf{Riemann scheme:} \[\begin{Bmatrix}
0&1&s&\infty\\[0.1cm] \hline\\[-0.25cm]\begin{array}{c} 0
\\\noalign{\medskip}0
\\\noalign{\medskip}0
\\\noalign{\medskip}0
\\\noalign{\medskip} \end{array}&\begin{array}{c} 0
\\\noalign{\medskip}1
\\\noalign{\medskip}1
\\\noalign{\medskip}2
\\\noalign{\medskip} \end{array}&\begin{array}{c} 0
\\\noalign{\medskip}1
\\\noalign{\medskip}1
\\\noalign{\medskip}2
\\\noalign{\medskip}\end{array}&\begin{array}{c} \frac{1}{2}
\\\noalign{\medskip}1-\lambda
\\\noalign{\medskip}1+\lambda
\\\noalign{\medskip}\frac{3}{2}
\\\noalign{\medskip}\end{array}
\end{Bmatrix}.\]

\textbf{CY-operators:} For $\lambda=\frac{1}{2}$ we get operators of CY-type, which we have already constructed before and thus omit them in the following table. Despite that we find

\begin{longtable}{c c c}
\textbf{Number}&\textbf{Operator}&\textbf{Source}\\\hline\hline
$1'$&\begin{minipage}{1cm}\begin{align*}4\,{\vartheta}^{4}-2\,z \left( 2\,\vartheta+1 \right) ^{2} \left( 7\,{
\vartheta}^{2}+7\,\vartheta+3 \right)\\ +81\,{z}^{2} \left( 2\,\vartheta
+1 \right)  \left( \vartheta+1 \right) ^{2} \left( 2\,\vartheta+3
 \right) 
\end{align*}\end{minipage} 
&$2, -7, -104$\\

$2'$&\begin{minipage}{1cm}\begin{align*}4\,{\vartheta}^{4}-2\,z \left( 2\,\vartheta+1 \right) ^{2} \left( 11\,
{\vartheta}^{2}+11\,\vartheta+5 \right)\\ +125\,{z}^{2} \left( 2\,
\vartheta+1 \right)  \left( \vartheta+1 \right) ^{2} \left( 2\,
\vartheta+3 \right) 
\end{align*}\end{minipage} 
&$2, 4, -8$\\

$3'$&\begin{minipage}{1cm}\begin{align*}{\vartheta}^{4}-z \left( 2\,\vartheta+1 \right) ^{2} \left( 5\,{
\vartheta}^{2}+5\,\vartheta+2 \right)\\ +16\,{z}^{2} \left( 2\,\vartheta
+1 \right)  \left( \vartheta+1 \right) ^{2} \left( 2\,\vartheta+3
 \right) 
\end{align*}\end{minipage} 
&$4, 20, 644/3$\\

$4'$&\begin{minipage}{1cm}\begin{align*}{\vartheta}^{4}-2\,z \left( 2\,\vartheta+1 \right) ^{2} \left( 3\,{
\vartheta}^{2}+3\,\vartheta+1 \right)\\ +4\,{z}^{2} \left( 2\,\vartheta+
1 \right)  \left( \vartheta+1 \right) ^{2} \left( 2\,\vartheta+3
 \right) 
\end{align*}\end{minipage} 
&$8, 63, 1000$\\

$5'$&\begin{minipage}{1cm}\begin{align*}4\,{\vartheta}^{4}+6\,z \left( 2\,\vartheta+1 \right) ^{2} \left( 3\,
{\vartheta}^{2}+3\,\vartheta+1 \right) \\-27\,{z}^{2} \left( 2\,
\vartheta+1 \right)  \left( \vartheta+1 \right) ^{2} \left( 2\,
\vartheta+3 \right) 
\end{align*}\end{minipage} 
&$6, 93/2, 608$\\

$6'$&\begin{minipage}{1cm}\begin{align*}4\,{\vartheta}^{4}-2\,z \left( 2\,\vartheta+1 \right) ^{2} \left( 17\,
{\vartheta}^{2}+17\,\vartheta+5 \right)\\ +{z}^{2} \left( 2\,\vartheta+1
 \right)  \left( \vartheta+1 \right) ^{2} \left( 2\,\vartheta+3
 \right) 
\end{align*}\end{minipage} 
&$14, 303/2, 10424/3$\\

$7'$&\begin{minipage}{1cm}\begin{align*}4\,{\vartheta}^{4}+2\,z \left( 2\,\vartheta+1 \right) ^{2} \left( 13
\,{\vartheta}^{2}+13\,\vartheta+4 \right)\\ -3\,{z}^{2} \left( 2\,
\vartheta+1 \right)  \left( 3\,\vartheta+2 \right)  \left( 3\,
\vartheta+4 \right)  \left( 2\,\vartheta+3 \right) 
\end{align*}\end{minipage} 
&$10, 191/2, 1724$\\

$8'$&\begin{minipage}{1cm}\begin{align*}{\vartheta}^{4}-8\,z \left( 2\,\vartheta+1 \right) ^{2} \left( 5\,{
\vartheta}^{2}+5\,\vartheta+2 \right)\\ +192\,{z}^{2} \left( 2\,\vartheta
+1 \right)  \left( 3\,\vartheta+2 \right)  \left( 3\,\vartheta+4
 \right)  \left( 2\,\vartheta+3 \right) 
\end{align*}\end{minipage} 
&strange\\

$9'$&\begin{minipage}{1cm}\begin{align*}{\vartheta}^{4}-4\,z \left( 2\,\vartheta+1 \right) ^{2} \left( 3\,{
\vartheta}^{2}+3\,\vartheta+1 \right)\\ -16\,{z}^{2} \left( 2\,\vartheta
+1 \right)  \left( 4\,\vartheta+3 \right)  \left( 4\,\vartheta+5
 \right)  \left( 2\,\vartheta+3 \right) 
\end{align*}\end{minipage}&$4, 39, 364$\\

$10'$&\begin{minipage}{1cm}\begin{align*}{\vartheta}^{4}-4\,z \left( 2\,\vartheta+1 \right) ^{2} \left( 7\,{
\vartheta}^{2}+7\,\vartheta+3 \right)\\ +48\,{z}^{2} \left( 2\,\vartheta
+1 \right)  \left( 4\,\vartheta+3 \right)  \left( 4\,\vartheta+5
 \right)  \left( 2\,\vartheta+3 \right) 
\end{align*}\end{minipage}&$4, 7, 556/9$\\
\end{longtable}

Note that if we start with $8'$, we end up with an operator which seems to fulfill (Q) and (S) but admits no genus zero instanton numbers in $\mathbb{Z}\left[\frac{1}{N}\right]$. Actually, this is the only known example of this type.

\textbf{Family:}
\begin{align*}
&Q_3(s_1,s_2,\lambda,c)={\vartheta}^{4} \left( s_{{1}}-s_{{2}} \right) ^{4}-z \left( s_{{1}}-s
_{{2}} \right) ^{2} ((s^2_1-10s_1s_2+s^2_2)(\vartheta^4+2\vartheta^3))\\&-z \left( s_{{1}}-s
_{{2}} \right) ^{2}\vartheta^2\left((s_1^2+s_2^2)(1+\lambda-\lambda^2)+2s_1s_2(\lambda^2+\lambda-12)+2c(s_1+s_2)\right)
\\&-z \left( s_{{1}}-s
_{{2}} \right) ^{2}\vartheta\left((s_1^2+s_2^2)\lambda(1-\lambda)+2s_1s_2(\lambda^2+\lambda-7)+2c(s_1+s_2)\right)
\\&-z \left( s_{{1}}-s
_{{2}} \right) ^{2}(s_1s_2(\lambda^2-3)+c\lambda(s_1+s_2)+c^2)
\\&-4z^2s_1s_2(\vartheta+1)^2\left(2(s^2_1-4s_1s_2+s^2_2)(\vartheta^2+2\vartheta)\right)\\&-4z^2s_1s_2(\vartheta+1)^2\left((s_1^2+s_2^2)(3+\lambda-2\lambda^2)+s_1s_2(4\lambda^2+2\lambda-13)+2c(s_1+s_2)\right)
\\&-4z^3s_1^2s_2^2(\vartheta+1)(\vartheta+2)(2\vartheta+3+2\lambda)(2\vartheta+3-2\lambda)
\end{align*}

\textbf{Riemann scheme:} \[\begin{Bmatrix}
0&1&-\frac{(s-1)^2}{4s}&\infty\\[0.1cm] \hline\\[-0.25cm]\begin{array}{c} 0
\\\noalign{\medskip}0
\\\noalign{\medskip}0
\\\noalign{\medskip}0
\\\noalign{\medskip} \end{array}&\begin{array}{c} 0
\\\noalign{\medskip}1
\\\noalign{\medskip}1
\\\noalign{\medskip}2
\\\noalign{\medskip}\end{array}
&\begin{array}{c} -\frac{1}{2}
\\\noalign{\medskip}0
\\\noalign{\medskip}1
\\\noalign{\medskip}\frac{3}{2}\\\noalign{\medskip}\end{array}&\begin{array}{c} 1
\\\noalign{\medskip}\frac{3}{2}-\lambda
\\\noalign{\medskip}\frac{3}{2}+\lambda
\\\noalign{\medskip}2
\\\noalign{\medskip} \end{array}
\end{Bmatrix}\]

\textbf{CY-operators:} \begin{longtable}{c c c}
\textbf{Number}&\textbf{Operator}&\textbf{Source}\\\hline\hline
$1$&\begin{minipage}{1cm}\begin{align*}{\vartheta}^{4}-z \left( 28+145\,{\vartheta}^{4}+290\,{\vartheta}^{3}+
285\,{\vartheta}^{2}+140\,\vartheta \right)\\ +32\,{z}^{2} \left( 194\,{
\vartheta}^{2}+388\,\vartheta+327 \right)  \left( \vartheta+1 \right) 
^{2}\\-20736\,{z}^{3} \left( \vartheta+2 \right)  \left( \vartheta+1
 \right)  \left( 2\,\vartheta+3 \right) ^{2}
\end{align*}\end{minipage}&\textit{5, 42, 454}\\

$2$&\begin{minipage}{1cm}\begin{align*}{\vartheta}^{4}-z \left( 19\,{\vartheta}^{2}+19\,\vartheta+6 \right) 
 \left( 7\,{\vartheta}^{2}+7\,\vartheta+2 \right)\\ +8\,{z}^{2} \left( 
127\,{\vartheta}^{2}+254\,\vartheta+224 \right)  \left( \vartheta+1
 \right) ^{2}\\-500\,{z}^{3} \left( \vartheta+2 \right)  \left( 
\vartheta+1 \right)  \left( 2\,\vartheta+3 \right) ^{2}
\end{align*}\end{minipage}&\textit{13, 128, 2650}\\

$3$&\begin{minipage}{1cm}\begin{align*}{\vartheta}^{4}+2z \left( 4\,{\vartheta}^{4}+8\,{\vartheta}^{3}+37
\,{\vartheta}^{2}+33\,\vartheta+9 \right)\\ -36\,{z}^{2} \left( 92\,{
\vartheta}^{2}+184\,\vartheta+189 \right)  \left( \vartheta+1 \right) 
^{2}\\-20736\,{z}^{3} \left( \vartheta+2 \right)  \left( \vartheta+1
 \right)  \left( 2\,\vartheta+3 \right) ^{2}
 \end{align*}\end{minipage}&\textit{10, 24, 664}\\

$4$&\begin{minipage}{1cm}\begin{align*}{\vartheta}^{4}+z \left( 80+240\,{\vartheta}^{4}+480\,{\vartheta}^{3}+
592\,{\vartheta}^{2}+352\,\vartheta \right)\\ +2048\,{z}^{2} \left( 6\,{
\vartheta}^{2}+12\,\vartheta+5 \right)  \left( \vartheta+1 \right) ^{2
}\\-65536\,{z}^{3} \left( \vartheta+2 \right)  \left( \vartheta+1
 \right)  \left( 2\,\vartheta+3 \right) ^{2}
\end{align*}\end{minipage}&\textit{16, -110, 1744}\\

$5$&\begin{minipage}{1cm}\begin{align*}{\vartheta}^{4}+z \left( 72+243\,{\vartheta}^{4}+486\,{\vartheta}^{3}+
567\,{\vartheta}^{2}+324\,\vartheta \right)\\ +5832\,{z}^{2} \left( 3\,{
\vartheta}^{2}+6\,\vartheta+4 \right)  \left( \vartheta+1 \right) ^{2}
\\+78732\,{z}^{3} \left( \vartheta+2 \right)  \left( \vartheta+1
 \right)  \left( 2\,\vartheta+3 \right) ^{2}
\end{align*}\end{minipage}&\textit{9, -72, 748}\\

$6$&\begin{minipage}{1cm}\begin{align*}{\vartheta}^{4}+z \left( 180+575\,{\vartheta}^{4}+1150\,{\vartheta}^{3
}+1379\,{\vartheta}^{2}+804\,\vartheta \right)\\ +3168\,{z}^{2} \left( 
26\,{\vartheta}^{2}+52\,\vartheta+27 \right)  \left( \vartheta+1
 \right) ^{2}\\-20736\,{z}^{3} \left( \vartheta+2 \right)  \left( 
\vartheta+1 \right)  \left( 2\,\vartheta+3 \right) ^{2}
\end{align*}\end{minipage}&\textit{37, -570, 15270}\\

$1'$&\begin{minipage}{1cm}\begin{align*}{\vartheta}^{4}+z \left( 210+776\,{\vartheta}^{4}+1552\,{\vartheta}^{3
}+1738\,{\vartheta}^{2}+962\,\vartheta \right)\\ +324\,{z}^{2} \left( 
580\,{\vartheta}^{2}+1160\,\vartheta+747 \right)  \left( \vartheta+1
 \right) ^{2}\\+13436928\,{z}^{3} \left( \vartheta+1 \right) ^{2}
 \left( \vartheta+2 \right) ^{2}
\end{align*}\end{minipage}&\textbf{10, -872, 18328}\\

$2'$&\begin{minipage}{1cm}\begin{align*}{\vartheta}^{4}+z \left( 310+1016\,{\vartheta}^{4}+2032\,{\vartheta}^{
3}+2410\,{\vartheta}^{2}+1394\,\vartheta \right)\\ +500\,{z}^{2} \left( 
532\,{\vartheta}^{2}+1064\,\vartheta+563 \right)  \left( \vartheta+1
 \right) ^{2}\\+4000000\,{z}^{3} \left( \vartheta+1 \right) ^{2} \left( 
\vartheta+2 \right) ^{2}\end{align*}
\end{minipage}&\textbf{66, -1780, 69048}\\

$3'$&\begin{minipage}{1cm}\begin{align*}{\vartheta}^{4}+z \left( 152+368\,{\vartheta}^{4}+736\,{\vartheta}^{3}
+1020\,{\vartheta}^{2}+652\,\vartheta \right)\\ -8192\,{z}^{2} \left( {
\vartheta}^{2}+2\,\vartheta+6 \right)  \left( \vartheta+1 \right) ^{2}
\\-9437184\,{z}^{3} \left( \vartheta+1 \right) ^{2} \left( \vartheta+2
 \right) ^{2}
\end{align*}\end{minipage}&\textbf{68, 204, 125636/3}\\

$4'$&\begin{minipage}{1cm}\begin{align*}{\vartheta}^{4}-16\,z \left( 6\,{\vartheta}^{2}+6\,\vartheta-1
 \right)  \left( 2\,\vartheta+1 \right) ^{2}\\-1024\,{z}^{2} \left( 60\,
{\vartheta}^{2}+120\,\vartheta+97 \right)  \left( \vartheta+1 \right) 
^{2}\\-2097152\,{z}^{3} \left( \vartheta+1 \right) ^{2} \left( \vartheta
+2 \right) ^{2}
\end{align*}\end{minipage}&\textit{128,4084,382592}\\

$5'$&\begin{minipage}{1cm}\begin{align*}{\vartheta}^{4}-18\,z \left( 12\,{\vartheta}^{2}+12\,\vartheta+5
 \right)  \left( 3\,{\vartheta}^{2}+3\,\vartheta+1 \right)\\ +2916\,{z}^
{2} \left( 36\,{\vartheta}^{2}+72\,\vartheta+55 \right)  \left( 
\vartheta+1 \right) ^{2}\\-5038848\,{z}^{3} \left( \vartheta+1 \right) ^
{2} \left( \vartheta+2 \right) ^{2}
\end{align*}\end{minipage}&\textit{90, 2196, 151648}\\

$6'$&\begin{minipage}{1cm}\begin{align*}{\vartheta}^{4}+z \left( -46-1144\,{\vartheta}^{4}-2288\,{\vartheta}^{
3}-1590\,{\vartheta}^{2}-446\,\vartheta \right)\\ -4\,{z}^{2} \left( 
2300\,{\vartheta}^{2}+4600\,\vartheta+3621 \right)  \left( \vartheta+1
 \right) ^{2}\\-18432\,{z}^{3} \left( \vartheta+1 \right) ^{2} \left( 
\vartheta+2 \right) ^{2}
\end{align*}\end{minipage}&\textbf{266, 19320, 11433160/3}\\

$7$&\begin{minipage}{1cm}\begin{align*}{\vartheta}^{4}+z \left( 126+419\,{\vartheta}^{4}+838\,{\vartheta}^{3}
+985\,{\vartheta}^{2}+566\,\vartheta \right)\\ +196\,{z}^{2} \left( 250
\,{\vartheta}^{2}+500\,\vartheta+301 \right)  \left( \vartheta+1
 \right) ^{2}\\+28812\,{z}^{3} \left( 6\,\vartheta+11 \right)  \left( 6
\,\vartheta+7 \right)  \left( \vartheta+2 \right)  \left( \vartheta+1
 \right) 
\end{align*}\end{minipage}&\textbf{19, -276, 4455}\\

$8$&\begin{minipage}{1cm}\begin{align*}{\vartheta}^{4}+z \left( -24-248\,{\vartheta}^{4}-496\,{\vartheta}^{3}
-400\,{\vartheta}^{2}-152\,\vartheta \right)\\ +64\,{z}^{2} \left( 112\,
{\vartheta}^{2}+224\,\vartheta+187 \right)  \left( \vartheta+1
 \right) ^{2}\\-1536\,{z}^{3} \left( 6\,\vartheta+11 \right)  \left( 6\,
\vartheta+7 \right)  \left( \vartheta+2 \right)  \left( \vartheta+1
 \right) 
\end{align*}\end{minipage}&\textbf{32, 460, 16288}\\

$9$&\begin{minipage}{1cm}\begin{align*}{\vartheta}^{4}+z \left( 70+264\,{\vartheta}^{4}+528\,{\vartheta}^{3}+
586\,{\vartheta}^{2}+322\,\vartheta \right)\\ +100\,{z}^{2} \left( 228\,
{\vartheta}^{2}+456\,\vartheta+335 \right)  \left( \vartheta+1
 \right) ^{2}\\+40000\,{z}^{3} \left( 4\,\vartheta+5 \right)  \left( 4\,
\vartheta+7 \right)  \left( \vartheta+2 \right)  \left( \vartheta+1
 \right) 
\end{align*}\end{minipage}&\textbf{2, -44, 440}\\

$10$&\begin{minipage}{1cm}\begin{align*}{\vartheta}^{4}-2\,z \left( 23\,{\vartheta}^{2}+23\,\vartheta+5
 \right)  \left( 2\,\vartheta+1 \right) ^{2}\\-4\,{z}^{2} \left( 380\,{
\vartheta}^{2}+760\,\vartheta+657 \right)  \left( \vartheta+1 \right) 
^{2}\\-192\,{z}^{3} \left( 4\,\vartheta+5 \right)  \left( 4\,\vartheta+7
 \right)  \left( \vartheta+2 \right)  \left( \vartheta+1 \right) 
\end{align*}\end{minipage}&\textit{26, 348, 103520/9}\\

\end{longtable}

\textbf{Family:}
\begin{align*}
&Q_4(s_1,s_2,c,\lambda)=4\vartheta^4(s_1 -s_2)^4-z(s_1-s_2)^2\left(4(s_1^2-10s_1s_2+s_2^2)(\vartheta^4+2\vartheta^3)\right)\\& -z(s_1-s_2)^2\vartheta^2\left(5(s_1^2+s_2^2)+2s_1s_2(4\lambda^2-45)+8c(s_1+s_2)\right)\\&-z(s_1-s_2)^2\vartheta\left(s_1^2+s_2^2+2s_1s_2(4\lambda^2-25)+8c(s_1+s_2)\right)\\&-z(s_1-s_2)^2\left(s_1s_2(3\lambda^2-11)+2c(s_1+s_2)+4c^2
\right)\\&-4z^2s_1s_2(2\vartheta+2+\lambda)(2\vartheta+2-\lambda)(2(s_1^2-4s_1s_2+s_2^2)(\vartheta^2+2\vartheta))\\&-4z^2s_1s_2(2\vartheta+2+\lambda)(2\vartheta+2-\lambda)
(3(s_1^2+s_2^2)+s_1s_2(\lambda^2-11)+2c(s_1+s_2))
\\&-4z^3s_1^2s_2^2(2\vartheta+2+\lambda)(2\vartheta+2-\lambda)(2\vartheta+4+\lambda)(2\vartheta+4-\lambda)
\end{align*}

\textbf{Riemann scheme:} \[\begin{Bmatrix}
0&1&-\frac{(s-1)^2}{4s}&\infty\\[0.1cm] \hline\\[-0.25cm]\begin{array}{c} 0
\\\noalign{\medskip}0
\\\noalign{\medskip}0
\\\noalign{\medskip}0
\\\noalign{\medskip} \end{array}&\begin{array}{c} 0
\\\noalign{\medskip}1
\\\noalign{\medskip}1
\\\noalign{\medskip}2
\\\noalign{\medskip} \end{array}&\begin{array}{c} -\frac{1}{2}
\\\noalign{\medskip}0
\\\noalign{\medskip}1
\\\noalign{\medskip}\frac{3}{2}
\\\noalign{\medskip}\end{array}&
\begin{array}{c} 1-\frac{\lambda}{2}
\\\noalign{\medskip}1+\frac{\lambda}{2}
\\\noalign{\medskip}2-\frac{\lambda}{2}
\\\noalign{\medskip}2+\frac{\lambda}{2}
\\\noalign{\medskip}\end{array}
\end{Bmatrix}.\]

\textbf{CY-operators:}
\begin{longtable}{c c c}
\textbf{Number}&\textbf{Operator}&\textbf{Source}\\\hline\hline

$1''$&\begin{minipage}{1cm}\begin{align*}
{\vartheta}^{4}-z \left(2320\,{\vartheta}^{4}+4640\,{\vartheta}^
{3}+4228\,{\vartheta}^{2}+1908\,\vartheta+360 \right)\\ +1024\,{z}^{2}
 \left( 4\,\vartheta+5 \right)  \left( 4\,\vartheta+3 \right)  \left( 
97\,{\vartheta}^{2}+194\,\vartheta+144 \right)\\ -1327104\,{z}^{3}
 \left( 4\,\vartheta+5 \right)  \left( 4\,\vartheta+7 \right)  \left( 
4\,\vartheta+3 \right)  \left( 4\,\vartheta+9 \right) 
\end{align*}\end{minipage}&\textit{196, 17212, 2993772}\\

$2''$&\begin{minipage}{1cm}\begin{align*}
{\vartheta}^{4}-z \left(2128\,{\vartheta}^{4}+4256\,{\vartheta}^
{3}+3076\,{\vartheta}^{2}+948\,\vartheta+116 \right)\\ +16\,{z}^{2} \left( 4
\,\vartheta+5 \right)  \left( 4\,\vartheta+3 \right)  \left( 1016\,{
\vartheta}^{2}+2032\,\vartheta+1585 \right)\\ -32000\,{z}^{3} \left( 4\,
\vartheta+5 \right)  \left( 4\,\vartheta+7 \right)  \left( 4\,
\vartheta+3 \right)  \left( 4\,\vartheta+9 \right) 
\end{align*}\end{minipage}&\textit{-444, 57215, -19050964}\\

$3''$&\begin{minipage}{1cm}\begin{align*}
{\vartheta}^{4}+z \left(128\,{\vartheta}^{4}+256\,{\vartheta}^{3}
+1288\,{\vartheta}^{2}+1160\,\vartheta+300 \right)\\ -144\,{z}^{2} \left( 4
\,\vartheta+5 \right)  \left( 4\,\vartheta+3 \right)  \left( 368\,{
\vartheta}^{2}+736\,\vartheta+657 \right)\\ -1327104\,{z}^{3} \left( 4\,
\vartheta+5 \right)  \left( 4\,\vartheta+7 \right)  \left( 4\,
\vartheta+3 \right)  \left( 4\,\vartheta+9 \right) 
\end{align*}\end{minipage}&\textit{352, 18676, 15001120/3}\\

$4''$&\begin{minipage}{1cm}\begin{align*}
{\vartheta}^{4}+z \left(3840\,{\vartheta}^{4}+7680\,{\vartheta}^
{3}+9280\,{\vartheta}^{2}+5440\,\vartheta +1200\right)\\ +32768\,{z}^{2}
 \left( 4\,\vartheta+3 \right)  \left( 4\,\vartheta+5 \right)  \left( 
6\,{\vartheta}^{2}+12\,\vartheta+5 \right)\\ -4194304\,{z}^{3} \left( 4
\,\vartheta+5 \right)  \left( 4\,\vartheta+7 \right)  \left( 4\,
\vartheta+3 \right)  \left( 4\,\vartheta+9 \right) 
\end{align*}\end{minipage}&\textbf{480, -16536, 4215904}\\

$5''$&\begin{minipage}{1cm}\begin{align*}
{\vartheta}^{4}-z \left(3888\,{\vartheta}^{4}+7776\,{\vartheta}
^{3}+8748\,{\vartheta}^{2}+4860\,\vartheta+1044 \right)\\ +11664\,{z}^{2}
 \left( 4\,\vartheta+3 \right)  \left( 4\,\vartheta+5 \right)  \left( 
24\,{\vartheta}^{2}+48\,\vartheta+29 \right)\\ -5038848\,{z}^{3} \left( 
4\,\vartheta+5 \right)  \left( 4\,\vartheta+7 \right)  \left( 4\,
\vartheta+3 \right)  \left( 4\,\vartheta+9 \right) 
\end{align*}\end{minipage}&\textit{252,-19512,1162036}\\

$6''$&\begin{minipage}{1cm}\begin{align*}
{\vartheta}^{4}+z \left(9200\,{\vartheta}^{4}+18400\,{\vartheta}
^{3}+21628\,{\vartheta}^{2}+12428\,\vartheta+2712 \right)\\ +9216\,{z}^{2}
 \left( 4\,\vartheta+5 \right)  \left( 4\,\vartheta+3 \right)  \left( 
143\,{\vartheta}^{2}+286\,\vartheta+144 \right)\\ -1327104\,{z}^{3}
 \left( 4\,\vartheta+5 \right)  \left( 4\,\vartheta+7 \right)  \left( 
4\,\vartheta+3 \right)  \left( 4\,\vartheta+9 \right) 
\end{align*}\end{minipage}&\textbf{964, -111140, 85888580/3}\\

$7'$&\begin{minipage}{1cm}\begin{align*}
{\vartheta}^{4}-z \left(1000\,{\vartheta}^{4}+2000\,{\vartheta}^
{3}+1618\,{\vartheta}^{2}+618\,\vartheta +102\right)\\ +12\,{z}^{2} \left( 6
\,\vartheta+5 \right)  \left( 6\,\vartheta+7 \right)  \left( 419\,{
\vartheta}^{2}+838\,\vartheta+647 \right)\\ -7056\,{z}^{3} \left( 6\,
\vartheta+5 \right)  \left( 6\,\vartheta+13 \right)  \left( 6\,
\vartheta+7 \right)  \left( 6\,\vartheta+11 \right) 
\end{align*}\end{minipage}&\textbf{166, 8076, 1016100}\\

$8'$&\begin{minipage}{1cm}\begin{align*}
{\vartheta}^{4}-z \left(1792\,{\vartheta}^{4}+3584\,{\vartheta}^
{3}+4192\,{\vartheta}^{2}+2400\,\vartheta+528 \right)\\ +768\,{z}^{2}
 \left( 6\,\vartheta+5 \right)  \left( 6\,\vartheta+7 \right)  \left( 
31\,{\vartheta}^{2}+62\,\vartheta+34 \right)\\ -36864\,{z}^{3} \left( 6
\,\vartheta+5 \right)  \left( 6\,\vartheta+13 \right)  \left( 6\,
\vartheta+7 \right)  \left( 6\,\vartheta+11 \right) 
\end{align*}\end{minipage}&\textbf{-128, -5148, -263808}\\

$9'$&\begin{minipage}{1cm}\begin{align*}
{\vartheta}^{4}-z \left(912\,{\vartheta}^{4}+1824\,{\vartheta}^{
3}+1796\,{\vartheta}^{2}+884\,\vartheta+180 \right)\\ +176\,{z}^{2} \left( 8
\,\vartheta+7 \right)  \left( 8\,\vartheta+9 \right)  \left( 24\,{
\vartheta}^{2}+48\,\vartheta+35 \right)\\ -6400\,{z}^{3} \left( 8\,
\vartheta+17 \right)  \left( 8\,\vartheta+9 \right)  \left( 8\,
\vartheta+15 \right)  \left( 8\,\vartheta+7 \right) 
\end{align*}\end{minipage}&\textbf{60, 960, 61780}\\

$10'$&\begin{minipage}{1cm}\begin{align*}
{\vartheta}^{4}+z \left(1520\,{\vartheta}^{4}+3040\,{\vartheta}^{
3}+3628\,{\vartheta}^{2}+2108\,\vartheta +468\right)\\ +48\,{z}^{2} \left( 8
\,\vartheta+9 \right)  \left( 8\,\vartheta+7 \right)  \left( 184\,{
\vartheta}^{2}+368\,\vartheta+183 \right)\\ -2304\,{z}^{3} \left( 8\,
\vartheta+17 \right)  \left( 8\,\vartheta+9 \right)  \left( 8\,
\vartheta+15 \right)  \left( 8\,\vartheta+7 \right) 
\end{align*}\end{minipage}&\textbf{124, -3752, 2152276/9}\\
\end{longtable}

\textbf{Family:}
\begin{align*}
&Q_5(s_1,s_2,c,\lambda)=16s_1^2s_2^2\vartheta^4-4s_1s_2z(8(s_1+s_2)(\vartheta^4+2\vartheta^3))\\&-4s_1s_2z\vartheta^2\left(2(s_1+s_2)(a-a^2+9)+8c\right)\\&-4s_1s_2z\vartheta\left(2(s_1+s_2)(a-a^2+5)+8c
\right)\\&-4s_1s_2z\left((s_1+s_2)(a-a^2+2)+4c(a^2-a+1)
\right)\\&+z^2\left(16(s_1^2+4s_1s_2+s_2^2)(\vartheta^4+4\vartheta^3)\right)\\&+z^2\vartheta^2\left(4(s_1^2+s_2^2)(2a-2a^2+23)+32s_1s_2(a-a^2+15)+32c(s_1+s_2)\right)\\&+z^2\vartheta\left(8(s_1^2+s_2^2)(2a-2a^2+7)+64s_1s_2(a-a^2+7)+64c(s_1+s_2)\right)\\&+z^2\left((s_1^2+s_2^2)(a^4-2a^3-7a^2+8a+12)+2s_1s_2(84+20a-21a^2-a^4+2a^3)\right)\\&+z^2\left(8c(s_1+s_2)(a^2-a+4)+16c^2\right)\\&-2z^3(2\vartheta+3)^2\left(4(s_1+s_2)(\vartheta^2+3\vartheta)\right)\\&-2z^3(2\vartheta+3)^2\left((s_1+s_2)(3a-3a^2+11)+4c
\right)\\&+z^4(2\vartheta+3)(2\vartheta+5)(2\vartheta+3+2a)(2\vartheta+5-2a)
\end{align*}

\textbf{Riemann scheme:} \[\begin{Bmatrix}
0&1&s&\infty\\[0.1cm] \hline\\[-0.25cm]\begin{array}{c} 0
\\\noalign{\medskip}0
\\\noalign{\medskip}0
\\\noalign{\medskip}0
\\\noalign{\medskip} \end{array}&\begin{array}{c}-\frac{1}{2}
\\\noalign{\medskip}0
\\\noalign{\medskip}1
\\\noalign{\medskip}\frac{3}{2}
\\\noalign{\medskip} \end{array}&\begin{array}{c}-\frac{1}{2}
\\\noalign{\medskip}0
\\\noalign{\medskip}1
\\\noalign{\medskip}\frac{3}{2}
\\\noalign{\medskip}\end{array}&\begin{array}{c}\frac{3}{2}
\\\noalign{\medskip}\frac{5}{2}
\\\noalign{\medskip}\frac{3}{2}+a
\\\noalign{\medskip}\frac{5}{2}-a
\\\noalign{\medskip}\end{array}
\end{Bmatrix}\] 

\textbf{CY-operators:} We seem to get CY-operators if $a\in\left\{\frac{1}{2},\frac{1}{3},\frac{1}{4},\frac{1}{6}\right\}$. To be more precise, we find
\setlength{\extrarowheight}{0.5mm}
\begin{longtable}{c c l}
\textbf{Nr.}&\textbf{Operator}&\textbf{Source}\\\hline\hline
\multirow{7}{*}{$1'$}&{$16\vartheta^4-8z\left(56\vartheta^4+112\vartheta^3+(132+14a(1-a))\vartheta^2)\right)$}&{}\\[0.5mm]
&$-8z\left((76+14a(1-a))\vartheta+17+4a(1-a)\right)$&$a=\frac{1}{2}:\ -20, 199, 5924$\\[0.5mm]
&$+4z^2\left(1432\vartheta^4+5728\vartheta^3+(10670+716a(1-a))\vartheta^2\right)$&$a=\frac{1}{3}:\ -33, 1095/2, 29693$\\[0.5mm]
&$+4z^2\left((9884+1432a(1-a))\vartheta\right)$&$a=\frac{1}{4}:\ -76, 2958, 415420$\\[0.5mm]&$+4z^2\left(64a^3-32a^4-868a^2+836a+3681\right)$&$a=\frac{1}{6}:\ -492, 128514, 136094428$\\[0.5mm] 
&$-324\,{z}^{3} \left( 2\,\vartheta+3 \right) ^{2} \left( 
28\,{\vartheta}^{2}+84\,\vartheta+21\,a(1-a)+80 \right)$
&\\[0.5mm]
&$+6561\,
{z}^{4} \left( 2\,\vartheta+5 \right)  \left( 2\,\vartheta+3 \right) 
 \left( 2\,\vartheta+3+2\,a \right)  \left( 2\,\vartheta+5-2\,a
 \right)$&\\[0.5mm]\hline

\multirow{7}{*}{$2'$}&{$16\,{\vartheta}^{4}-8z \left(88\,{\vartheta}^{4}+176\,{\vartheta}^{
3}+2(104+11a(1-a))\vartheta^2\right)$}&\\[0.5mm]
&$-8z\left(2(60+11a(1-a))\vartheta+27+6a(1-a)\right)$&$a=\frac{1}{2}:\  28, -28, -1036$\\[0.5mm] 
&$+4{z}^{2} \left(2936\vartheta^4+11744\vartheta^3+(20822+1468a(1-a))\vartheta^2\right)$&$a=\frac{1}{3}:\ $\textit{45, -585/4, -7080}\\[0.5mm]
&$+4z^2\left((18156+2936a(1-a))\vartheta\right)$&$a=\frac{1}{4}:\ $\textit{100,-1285,-126580}\\[0.5mm]&$+4z^2\left(8a^3-4a^4-1612a^2+1608a+6417 
 \right)$&$a=\frac{1}{6}:\ $\textit{612, -87525, -51318900}\\[0.5mm]
&$ -500\,{z}^{3} \left( 2\,\vartheta+3 \right) ^{2} \left( 44\,{
\vartheta}^{2}+132\,\vartheta+33a(1-a) +126\right)$&\\[0.5mm] &$+15625\,{z
}^{4} \left( 2\,\vartheta+5 \right)  \left( 2\,\vartheta+3 \right) 
 \left( 2\,\vartheta+3+2\,a \right)  \left( 2\,\vartheta+5-2\,a
 \right)$&\\[0.5mm]\hline

\multirow{7}{*}{$3'$}&{$\vartheta^4-z(40\vartheta^4+80\vartheta^3+(94+10a(1-a))\vartheta^2)$}&\\[0.5mm]
&$-z((54+10a(1-a))\vartheta+12+3a(1-a))$&$a=\frac{1}{2}:\ -32, -284, -8736$\\[0.5mm]
&$+3z^2\left(176\vartheta^4+704\vartheta^3+(1188+88a(1-a))\vartheta^2\right)$&$a=\frac{1}{3}:\ -54, -864, -40552$\\[0.5mm]
&$+3z^2\left((968+176a(1-a))\vartheta\right)$&$a=\frac{1}{4}:\ -128, -5232, -1546624/3$\\[0.5mm]&$+3z^2(3a^4-6a^3-89a^2+92a+320)$&$a=\frac{1}{6}:\ -864, -261684, -147560800$\\[0.5mm]
&$-32\,{z}^{3} \left( 2\,
\vartheta+3 \right) ^{2} \left( 20\,{\vartheta}^{2}+60\,\vartheta+15\,
a(1-a)+57 \right)$&\\[0.5mm]
&$ +256\,{z}^{4} \left( 2\,\vartheta+5 \right) 
 \left( 2\,\vartheta+3 \right)  \left( 2\,\vartheta+3+2\,a \right) 
 \left( 2\,\vartheta+5-2\,a \right)$&\\[0.5mm]\hline 

\multirow{7}{*}{$4'$}&{$\vartheta^4-2z\left(24\vartheta^4+48\vartheta^3+(56+6a(1-a))\vartheta^2\right)$}&{}\\[0.5mm]
&$-2z\left((32+6a(1-a))\vartheta+7+2a(1-a)\right)$&$a=\frac{1}{2}:\ $\textit{48, -922, 32368}\\[0.5mm]
&$+4z^2\left(152\vartheta^4+608\vartheta^3+(926+76a(1-a))\vartheta^2\right)$&$a=\frac{1}{3}:\ $\textit{84, -5313/2, 148820}\\[0.5mm]
&$+4z^2\left((636+152a(1-a))\vartheta\right)$&$a=\frac{1}{4}:\ $\textit{208, -15150, 1863312}\\[0.5mm]&$+4z^2\left(8a^4-16a^3-64a^2+72a+169\right)$&$a=\frac{1}{6}:\ $\textit{1488, -705102, 517984144}\\[0.5mm]
&$-16\,{z}^{3} \left( 2\,
\vartheta+3 \right) ^{2} \left( 12\,{\vartheta}^{2}+36\,\vartheta-9\,{
a}^{2}+34+9\,a \right)$
&\\[0.5mm]&$ +16\,{z}^{4} \left( 2\,\vartheta+5 \right) 
 \left( 2\,\vartheta+3 \right)  \left( 2\,\vartheta+3+2\,a \right) 
 \left( 2\,\vartheta+5-2\,a \right)$&\\[0.5mm]\hline\newpage

\multirow{7}{*}{$5'$}&{$16\vartheta^4+24z\left(24\vartheta^4+48\vartheta^3+(56+6a(1-a))\vartheta)\right)$}&{}\\[0.5mm]
&$+24z\left((32+6a(1-a))\vartheta+7+2a(1-a)\right)$&$a=\frac{1}{2}:\ $\textit{36, -765, 62596/3}\\[0.5mm]
&$+4z^2\left(1080\vartheta^4+4320\vartheta^3+(5670+540a(1-a))\vartheta^2\right)$&$a=\frac{1}{3}:\ $\textit{63, -2205, 96866}\\[0.5mm]
&$+4z^2\left((2700+1080a(1-a))\vartheta\right)$&$a=\frac{1}{4}:\ $\textit{156, -12588, 1229332}\\[0.5mm]&$+4z^2\left(108a^4-216a^3-324a^2+432+225\right)$
&$a=\frac{1}{6}: $\textit{1116, -587268, 349462868}\\[0.5mm]&$-324\,{z}^{3} \left( 2\,\vartheta+3 \right) ^{2} \left( 12\,{
\vartheta}^{2}+36\,\vartheta+9a(1-a)+34\right)$&\\[0.5mm]
&$+729\,{z}^{4}
 \left( 2\,\vartheta+5 \right)  \left( 2\,\vartheta+3 \right)  \left( 
2\,\vartheta+3+2\,a \right)  \left( 2\,\vartheta+5-2\,a \right)$&\\[0.5mm]\hline

\multirow{7}{*}{$6'$}&{$16\vartheta^4-8z\left(136\vartheta^4+272\vartheta^3+(316+34a(1-a))\vartheta^2\right)$}&{}\\[0.5mm]
&$-8z\left((180+34a(1-a))\vartheta+39+12a(1-a)\right)$&$a=\frac{1}{2}:\ -76, -2002, -92996$\\[0.5mm]
&$+4z^2\left(4632\vartheta^4+18528\vartheta^3+(27342+2316a(1-a))\vartheta^2\right)$&$a=\frac{1}{3}:\ -135, -22815/4, -417685$\\[0.5mm]
&$+4z^2\left((17628+4632a(1-a))\vartheta\right)$&$a=\frac{1}{4}:\ -340, -31985, -15174100/3$\\[0.5mm]&$+4z^2\left(288a^4-576a^3-1860a^2+2148a+4209\right)$
&$a=\frac{1}{6}:\ -2484, -1446309, -1327731388$\\[0.5mm]&$-4\,{z}^{3} \left( 2\,\vartheta+3 \right) ^{2} \left( 68
\,{\vartheta}^{2}+204\,\vartheta+51a(1-a)+192 \right)$&\\[0.5mm]
&$ +{z}^{4
} \left( 2\,\vartheta+5 \right)  \left( 2\,\vartheta+3 \right) 
 \left( 2\,\vartheta+3+2\,a \right)  \left( 2\,\vartheta+5-2\,a
 \right)$&\\[0.5mm]\hline 
\end{longtable}

\bibliographystyle{amsalpha}

\bibliography{Literatur6}{}

\end{document}